\documentclass[12pt]{article}

\usepackage{amsmath,amsthm,amsfonts,txfonts,amssymb,amscd,epsf,epsfig,psfrag,enumerate,relsize}
\usepackage[mathscr]{eucal}
\usepackage{yhmath}
\usepackage{graphicx,epstopdf}
\usepackage{float}
\usepackage{wasysym}
\usepackage{scrextend}

\title{Laminations with transverse measures in ordered abelian semigroups}

\author{Ulrich Oertel}

\date{Revised July, 2016}

\newenvironment{tightenum}{
\begin{enumerate}[(i)]
  \setlength{\itemsep}{1pt}
  \setlength{\parskip}{0pt}
  \setlength{\parsep}{0pt}
}{\end{enumerate}}
\newcounter{zahl}
\newenvironment{intenum}{
\begin{enumerate}[(1)]
  \setlength{\itemsep}{1pt}
  \setlength{\parskip}{0pt}
  \setlength{\parsep}{0pt}
}{\end{enumerate}}
\newtheorem{thm}{Theorem}[section] \newtheorem{lemma}[thm]{Lemma}

\newtheorem{corollary}[thm]{Corollary}

\newtheorem{proposition}[thm]{Proposition}
 
\newtheorem{problem}[thm]{Problem} 
\newtheorem*{claim*}{Claim}

\newtheorem{conj}[thm]{Conjecture} 
 \theoremstyle{definition}
\newtheorem{defn}[thm]{Definition}
\newtheorem{defns}[thm]{Definitions} 
 \newtheorem{ex}[thm]{Example}
\newtheorem{exs}[thm]{Examples}
 
\newtheorem{remark}[thm]{Remark}
\theoremstyle{remark}

\begin{document}

\maketitle


\def\HDS{half-disk sum}

\def\Length{\text{Length}}

\def\lub{\text{lub}}
\def\Area{\text{Area}}
\def\Im{\text{Im}}
\def\im{\text{Im}}
\def\cl{\text{cl}}
\def\rel{\text{ rel }}
\def\irred{irreducible}
\def\half{spinal pair }
\def\spinal{\half}
\def\spinals{\halfs}
\def\halfs{spinal pairs }
\def\reals{\mathbb R}
\def\rationals{\mathbb Q}
\def\complex{\mathbb C}
\def\naturals{\mathbb N}
\def\integers{\mathbb Z}
\def\id{\text{id}}
\def\Chi{\raise1.5pt \hbox{$\chi$}}

\def\proj{P}
\def\hyp {\hbox {\rm {H \kern -2.8ex I}\kern 1.15ex}}

\def\Diff{\text{Diff}}

\def\weight#1#2#3{{#1}\raise2.5pt\hbox{$\centerdot$}\left({#2},{#3}\right)}
\def\intr{{\rm int}}
\def\inter{\ \raise4pt\hbox{$^\circ$}\kern -1.6ex}
\def\Cal{\cal}
\def\from{:}
\def\inverse{^{-1}}
\def\Max{{\rm Max}}
\def\Min{{\rm Min}}
\def\fr{{\rm fr}}
\def\embed{\hookrightarrow}
\def\Genus{{\rm Genus}}
\def\Z{Z}
\def\X{X}

\def\roster{\begin{enumerate}}
\def\endroster{\end{enumerate}}
\def\intersect{\cap}
\def\definition{\begin{defn}}
\def\enddefinition{\end{defn}}
\def\subhead{\subsection\{}
\def\theorem{thm}
\def\endsubhead{\}}
\def\head{\section\{}
\def\endhead{\}}
\def\example{\begin{ex}}
\def\endexample{\end{ex}}
\def\ves{\vs}
\def\mZ{{\mathbb Z}}
\def\M{M(\Phi)}
\def\bdry{\partial}
\def\hop{\vskip 0.15in}
\def\hip{\vskip0.1in}
\def\mathring{\inter}
\def\trip{\vskip 0.09in}
\def\PML{\mathscr{PML}}
\def\J{\mathscr{J}}
\def\G{\mathscr{G}}
\def\H{\mathscr{H}}
\def\C{\mathscr{C}}
\def\S{\mathscr{S}}
\def\T{\mathscr{T}}
\def\E{\mathscr{E}}
\def\K{\mathscr{K}}
\def\PL{\mathscr{L}}
\def\suchthat{|}
\newcommand\invlimit{\varprojlim}
\newcommand\congruent{\equiv}
\newcommand\modulo[1]{\pmod{#1}}
\def\ML{\mathscr{ML}}
\def\Stack{\mathscr{T}}
\def\M{\mathscr{M}}
\def\A{\mathscr{A}}
\def\union{\cup}
\def\atlas{\mathscr{A}}
\def\interior{\text{Int}}
\def\frontier{\text{Fr}}
\def\composed{\circ}
\def\GL{\mathscr{GL}}
\def\PC{\mathscr{PC}}
\def\PM{\mathscr{PM}}
\def\D{\mathscr{D}}
\def\Proj{\mathscr{P}}
\def\PW{\mathscr{PW}}
\def\FD{\mathscr{FD}}
\def\FDM{\mathscr{FDM}}
\def\FDW{\mathscr{FDW}}
\def\PFDW{\mathscr{PFDW}}
\def\PFD{\mathscr{PFD}}
\def\PIFD{\mathscr{PIFD}}
\def\PQW{\mathscr{PQW}}
\def\PW{\mathscr{PW}}
\def\WC{\mathscr{WC}}
\def\W{\mathscr{W}}
\def\WM{\mathscr{WM}}
\def\PWM{\mathscr{PWM}}
\def\PFDM{\mathscr{PFDM}}
\def\D{\mathscr{D}}
\def\B{\mathscr{B}}
\def\I{\mathscr{I}}
\def\level{\mathfrak{L}}
\def\levels{\mathscr{L}}
\def\real{\mathfrak{R}}
\def\residue{\mathfrak{R}}
\def\FDL{\mathscr{FDL}}
\def\FDM{\mathscr{FDM}}
\def\PFDL{\mathscr{PFDL}}
\def\FDLM{\mathscr{FDLM}}
\def\PFDLM{\mathscr{PFDLM}}
\def\CFD{\mathscr{QFD}}
\def\CFD{\mathscr{CFD}}
\def\PCFD{\mathscr{PCFD}}
\def\CFDM{\mathscr{CFDM}}
\def\PCFDM{\mathscr{PCFDM}}
\def\PCFD{\mathscr{PCFD}}
\def\FDT{\mathscr{FDT}}
\def\PFDT{\mathscr{PFDT}}
\def\PCFDT{\mathscr{PCFDT}}
\def\PCFD{\mathscr{PCFD}}
\def\PCV{\mathscr{PCV}}
\def\CV{\mathscr{CV}}

\def\WC{\mathscr{WC}}
\def\PI{\mathscr{PI}}
\def\PQI{\mathscr{PQI}}
\def\VM{\mathscr{VM}}
\def\PVM{\mathscr{PVM}}
\def\PLM{\mathscr{LM}}
\def\LMM{\mathscr{LMM}}
\def\DG{\breve{\mathscr{G}}}
\def\SG{\mathfrak{S}}
\def\V{\mathscr{V}}
\def\LV{\mathscr{LV}}
\def\PV{\mathscr{PV}}
\def\QV{\mathscr{QV}}
\def\PQV{\mathscr{PQV}}
\def\abs{\odot}
\def\DS{\breve S}
\def\DL{\breve L}
\def\DBL{\breve{\bar L}}
\def\II{[0,\infty]}
\def\equiv{\hskip -3pt \sim}

\def\split{\prec}
\def\pinch{\succ}
\def\OB{\mathbb O}
\def\FB{\mathbb F}
\def\SB{\mathbb S}
\def\AB{\mathbb A}
\def\PB{\mathbb P}
\def\WB{\mathbb W}
\def\LB{\mathbb L}
\def\KB{\mathbb K}
\def\DLB{\mathbb {DL}}
\def\DKB{\mathbb {DK}}
\def\DSB{\mathbb {DS}}
\def\DOB{\mathbb {DO}}

\def\Aut{\mbox{Aut}}

\def\TB{\mathbb T}
\def\OBB{{\mathbb S}\kern -6pt\raisebox{1.3pt}{--} \kern 2pt}
\def\Infty{\hbox{$\infty$\kern -8.1pt\raisebox{0.2pt}{--}\kern 1pt}}
\def\ens{\bar\circledwedge}
\def\ins{\circledwedge}
\def\sins{\circledvee}
\def\sens{\bar\circledvee}
\def\isom{\cong}
\def\bw{\bf w}
\def\star{\textasteriskcentered}
\newcommand{\bigins}{\mathop{\mathlarger{\mathlarger\circledwedge}}}
\newcommand{\bigsins}{\mathop{\mathlarger{\mathlarger\circledvee}}}

\vskip 0.3in
\begin{abstract} We describe  a construction of ordered algebraic structures (ordered abelian semigroups, ordered commutative semirings, etc.) and describe applications to codimension-1 laminations. For a suitable ordered semi- algebraic structure $\LB$ and measurable space $X$ we define $\LB$-measures $\nu$ on $X$.     If $L$ is a codimension-1 lamination in a manifold, it often admits transverse $\LB$-measures for some $\LB$.  Transverse $\LB$-measures can be used to understand classes of laminations much larger than the class of laminations  admitting transverse positive $\reals$-measures.  In particular, we show that ``finite or infinite depth measured laminations" are laminations admitting transverse measures with values in a certain ordered semiring $\bar\OB$ satisfying the additional property that locally the values lie in a smaller semiring $\PB$.  We consider the ``realization problem:" In one version, this deals with the problem whether an $\PB$-invariant weight vector assigned to a branched manifold $B$ (satisfying certain branch equations) determines a lamination $L$ carried by $B$ with a transverse $\bar\OB$-measure inducing the weights on $B$.  We describe further laminations which may not be $\LB$-measured, but are ``well-covered" by laminations with transverse $\LB$-measures.    We also investigate actions on $\LB$-trees which are associated to essential laminations with transverse $\LB$-measures. 

In appendices, we develop ideas about $\LB$-measures a little further, for example showing that a $\PB$-measure can be interpreted as a kind of probability measure.   \end{abstract}

\section{Introduction.}

In this paper we construct certain ordered semi- algebraic structures which probably have been described before.  The author would appreciate any references to the literature. Then we describe transverse measures for codimension-1 laminations with values in some of these ordered semi- algebraic structures.    In an appendix we describe ``probability measures" with values in certain ordered semifields.    

 One important example of a semi-algebraic structure is an ordered semifield which we call $\PB$.   This can be used for probability calculations and for encoding ``finite or infinite depth measured laminations."  The idea for constructing $\PB$ (and other semi-algebraic structures we consider) is to start with the non-negative reals, but then refine in such a way that 0 is replaced by another real number ``at a different level" which `` measures the size of a 0."  This can be repeated as often as one wishes.  For probability measures with values in $\PB$, this means that we consider probabilities of events whose conventional probability is 0, assigning a probability ``at a different level" which is non-zero.

Returning to laminations, for simplicity we first consider a closed compact manifold $M$.  Given a codimension -1 lamination $L$ in $M$, one can ask whether it has a transverse $\reals$-valued measure.   If so, the lamination can be described by finite data, consisting of finitely many (non-negative) $\reals$-weights assigned to sectors of a branched manifold $B\embed M$ ($B$ embedded in $M$), so that the lamination is in some sense completely understood in terms of these weights.   Here the ``sectors" are completions of components of the complement of the branch locus of the branched manifold, see Definition \ref{BranchedManifoldDef} for definitions related to branched manifolds..

 One goal is to understand more general classes of laminations in terms of transverse measures.  We achieve the goal by considering measures with values in certain ordered semi- algebraic structures.

 \begin{defn}\label{FiniteDepthLam}  A {\it finite depth measured lamination} in a manifold $M$ is an isotopy class of a codimension-1 laminations of the form $\displaystyle L=\bigcup_{j=-d}^0L_j$ where $(L_{-d},\ldots, L_{-1}, L_0)$  is a finite sequence of $\reals$-measured laminations, with each $L_i$ embedded in $\displaystyle M_i= M\setminus \bigcup_{j>i}L_j$, with  $\bigcup_{j>i}L_j$ being a lamination in $M$, and with $L_i$ having a full support transverse $\reals$-measure $\mu_i$.    
 
 An {\it infinite depth codimension-1 measured lamination}  is defined in the same way but with $\displaystyle L=\bigcup_{j=-\infty}^0L_j$.  A {\it bi-infinite multi-level measured lamination $L$} is defined in the same way with   $\displaystyle L=\bigcup_{j=-\infty}^\infty L_j$.  An {\it infinite height measured lamination} has the form $\displaystyle L=\bigcup_{j=0}^\infty L_j$.  In all cases $(L_i,\mu_i)$ is an $\reals$-measured lamination in $M_i=M\setminus \bigcup_{j>i}L_j$, and $\bigcup_{j>i}L_j$ is a lamination in $M$.  A {\it finite height measured lamination} is the same as a finite depth foliation.
 
 A lamination of any of the types described above is called a {\it multi-level measured lamination}.
 \end{defn}
 
 Examples of 1-dimensional finite depth measured laminations in a surface can be constructed from a sequence of measured laminations $L_i$, each $L_i$ embedded in the complement of $\cup_{j>i}L_j$.   We describe the simplest possible example which is not $\reals$-measured.
 
 \begin{example}  Let $M$ be the 2-dimensional torus, let $L_0$ be a simple closed curve in $M$, and let $L_{-1}$ be a single leaf spiraling towards the closed leaf from both sides.   Each of $L_0$ and $L_{-1}$ have a transverse atomic measure.  Then $L=L_0\cup L_{-1}$ is a finite depth measured lamination which does not admit an $\reals$-measure.
 
 \end{example}
 
\begin{example}  For an example of a bi-infinite multi-level measured lamination, we consider the branched surface $B$ shown in Figure \ref{Infinite}.  The branched surface is shown as an immersed branched surface, but it can be embedded in $\reals^3$, and we assume it is so embedded.  The sectors of the branched surface are labeled $\hat L_i$.  For example, the sector $\hat L_0$ is a pair of pants (or sphere with 3 holes).  The {\it leaf} $L_0$ is obtained by identifying the two boundary components of the sector $\hat L_0$ which are mapped to the same branched curve of $B$  and extending at the remaining boundary, following parallel to $\hat L_1$, etc.  Then the union of the leaves $L_i$ form a  bi-infinite multi-level measured lamination.   For a rigorous construction of the leaves (laminations) $L_i$, one must find a sequence of splittings (or a 1-parameter family of splittings) of the branched surface $B$ without introducing intersections of the branch locus, extending the surfaces labeled $L_i$.   Then the inverse limit of the split branched surfaces is a lamination with leaves labeled $L_i$.  Furthermore, one sees that the leaf $L_i$ only limits on higher level $L_j$, as required.  We shall describe the process of splitting more rigorously later.

 \begin{figure}[H]
\centering
\scalebox{1}{\includegraphics{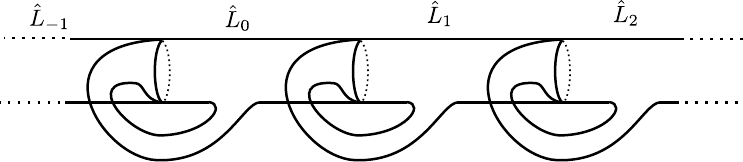}}
\caption{\small An example of a lamination $L$ with bi-infinite measured levels.}
\label{Infinite}
\end{figure}
\end{example}
 
 We give a definition of transverse measures for codimension-1 laminations with values in certain ordered algebraic structures.  We will not describe  such an ordered algebraic structure $\LB$ in the introduction; it is enough to know that for appropriate $\LB$, there are $\LB$-measures with values in $\LB$.   Suppose $L$ is a codimension-1 lamination in a manifold $M$.    Suppose $T$ is any transverse arc for $L$.  If $\LB$ is an ordered abelian semigroup, a {\it transverse $\LB$-measure} for $L$  assigns a measure $\mu(E)$ to each Borel set in $T$ so that the usual measure axioms are satisfied for $\mu$ a measure on $T$.  In addition, a transverse $\LB$-measure for $L$ must satisfy the {\it  invariance properties}:  (i) If $T'$ is a subarc of $T$, then the measure on $T'$ is the restriction of the measure on $T$.  (ii) If $T'$ is homotopic via a homotopy through transversals to $T$, then the homotopy takes the measure on $T'$ to the measure on $T$.  The homotopy of transversals should move each endpoint of the transversal such that it remains in a single leaf or remains in a component of the complement of $L$.

 In the special case that a lamination $L$ has a transverse measure $\mu$ with values in $\LB=[0,\infty]\subset \bar\reals$, we will often say that $(L,\mu)$ is an $\bar\reals$-measured lamination.  A $\bar \reals$-measured lamination is what is usually just called a ``measured lamination," for example in the Nielsen-Thurston theory of automorphisms of surfaces.   It is usually understood that a transverse $\bar\reals$-measure is locally finite, which means that for every transverse arc $T$, every $x\in T$ has a neighborhood $V\subset T$ with $\mu(V)<\infty$.  Another way of saying this is that $\mu$ is a Radon measure on $T$.  Clearly a Radon measure is finite on compact subsets of transversals and conversely if the measure on a transversal is finite on compact subsets of transversals, then it is locally finite.  Also, it is a standard fact that a Radon measure is inner regular.   In our context, the local finiteness property will be replaced by a different, related condition.   We will use two ordered commutative semirings $\PB$ and $\bar \OB$ which will be defined later, where $\PB\subset \bar\OB$ is actually a semifield . Suppose $L$ has a transverse $\bar\OB$-measure $\mu$, which yields a measure on every transversal $T$.   We shall work with measures which have further particular properties.  We will show that these measures are determined by their values in $\PB$, so we call them $(\PB,\bar\OB)$-measures.    
  
 We say a lamination $L\embed M$ {\it has parallel leaves} if there exist two leaves $\ell_0$ and $\ell_1$ such that the completion of $M\setminus(\ell_0\cup\ell_1)$ contains a product.

 \begin{thm} \label{FiniteDepthThm}     
 
 \noindent (a) A transverse Borel $(\PB,\bar\OB)$-measure on a given codimension-1 lamination $L$ in a manifold $M$ corresponds uniquely to a structure as as a multi-level measured lamination,  and conversely.
 
 \noindent (b) If $M$ is a compact surface (or a surface with cusps of finite type) and $L$ is essential without parallel leaves, then a transverse $(\PB,\bar\OB)$-measure on $L$ corresponds uniquely to a structure for $L$ as a finite depth measured lamination, and conversely.
\end{thm}

We will define {\it $\integers^r$ multi-level measured laminations} and we will show that such laminations can be represented as laminations with certain transverse measures called $(\PB_r,\bar\OB_r)$-measures.  The case $r=1$ gives us $(\PB,\bar\OB)$-measures.

\begin{thm}\label{RecursiveMultiThm}
 A full support transverse Borel $(\PB_r,\bar\OB_r)$-measure on a given codimension-1 lamination $L$ in a manifold $M$ corresponds uniquely to a structure as as an $\integers^r$ multi-level measured lamination,  and conversely.
\end{thm}

If $L$ is an essential lamination in a compact orientable closed surface $S$, $\chi(S)<0$, then there is an order tree $\T$ associated to the lift $\tilde L$ in the universal cover $\tilde S$ of $S$.   If $L$ admits a transverse $\LB$-measure, for some ordered abelian semigroup $\LB$, then $\T$ also obtains an additional structure related to $\LB$.   We state a simplified version of Proposition \ref{DualMeasureProp} here in the introduction:

\begin{proposition}\label{DualMeasureProp} Suppose $S$ is a compact orientable closed surface or a surface with $\chi(S)<0$. Given an essential lamination $L$ in $S$ which admits a transverse $\LB$-measure $\mu$, the associated order tree $\T$ is $\LB$-measured: there is an $\LB$-measure $\nu$ on the disjoint union of its segments such that measures on two different segments agree on the intersection segment.  If $L$ has no leaves with atomic transverse measure, then $\T$ is an $\LB$-metric space with metric $d(x,y)=\nu([x,y])$ for $x,y\in \T$. 
\hip
\noindent   There is an action of $\pi_1(S)$ on $\T$ which preserves the metric and measure on $\T$ \end{proposition}

It should be easy to prove a proposition similar to the above for essential 2-dimensional laminations in a 3-manifold with transverse $\LB$-measures.

Given a codimension-1 branched manifold $B\embed M$ carrying an $\bar\reals$-measured $L$, the measure on $L$ induces a set of {\it invariant} $\reals$-weights on the sectors of $B$.  Conversely, given an invariant $\reals$-weight vector $\bw$ assigning a weight to each sector of $B$, the weight vector  determines an $\bar\reals$-measured lamination up to some minor ambiguity having to do with atomic measures.   The fact that these laminations can be represented so simply by finite data on a finite combinatorial object makes them easier to understand.   Similarly, if $B$ carries an $\LB$-measured lamination $L$, there is an induced invariant $\LB$-weight vector on $B$.   In general, producing an $\LB$-measured lamination from and $\LB$- invariant weight vector on a codimension-1 branched manifold may not be possible.   This is an important problem:

\begin{problem}[The realization problem.]\label{RealizeProblem}  Find good sufficient conditions for an invariant $\PB_r$-weight vector on a codimension-1 branched manifold $B\embed M$ to be realizable by a $(\PB_r,\bar\OB_r)$-measured lamination carried by $B$.
\end{problem}

The problem could also be stated for transverse measures with values in other ordered algebraic structures.  The first case to consider is the realization of invariant $(\PB,\bar\OB)$ invariant weight vectors.   The easiest result is 

 \begin{proposition} \label{TrackExistenceProp}.  Suppose $B\embed S$ is a train track embedded in a surface $S$ with an invariant $\PB$ weight vector $\bw$.  Then there exists a $(\PB,\OB)$-measured lamination $(L,\mu)$ carried by $B$ inducing the weight vector $\bw$.   The measured lamination $(L,\mu)$ is not unique, but there is a canonical choice for $(L,\mu)$.
 \end{proposition}
 
 For higher-dimensional manifolds and branched manifolds, we give a sufficient condition for the realization of $(\PB,\bar\OB)$ weight vectors in Theorem \ref{PORealizeThm}, and we give a sufficient condition for  $(\PB_r,\bar\OB_r)$ weight vectors in Theorem \ref{PrOrRealizeThm}.   In general, there is an obstruction to the realization of  $(\PB_r,\bar\OB_r)$ weight vectors, and our theorem is not nearly as strong as it might be.   A better realization theorem could lead to a powerful technique for analyzing codimension-1 laminations.   
 
 Part of the potential for our techniques comes from the fact that we can study not only laminations which have $(\PB_r,\bar\OB_r)$-measures, but also laminations which are {\it well-covered} by $(\PB_r,\bar\OB_r)$-measured laminations.  Well-covered laminations are explained in Section \ref{WellCoveredSection}.  Briefly, a lamination $L$ carried by $B\embed M$ is well-covered  by a $(\PB_r,\bar\OB_r)$-measured lamination $\tilde L\embed \tilde M$ in the universal cover $\tilde M$ of $M$ (or the universal cover of a  regular neighborhood of $B$) such that covering translations transform the $\bar\OB_r$-measure via order-isomorphisms of $\bar\OB_r$ coming from a kind of scalar multiplication of elements of $\bar\OB_r$ by elements of $\PB_r$.  These laminations $L$ can also be described using finite data on $B$.  The following is a very optimistic conjecture, which roughly says that any codimension-1 lamination can be approximated arbitrarily closely by a lamination which is well-covered  by a $(\PB_r,\bar\OB_r)$- measured lamination.
 
 \begin{conj}  Let $L$ be any codimension-1 lamination $L\embed M$ in a compact manifold $M$ carried by a branched manifold $B$.   Then there is also a lamination $L'$ carried by $B$ which is well-covered by a $(\PB_r,\bar\OB_r)$- measured lamination.
 \end{conj}
 
\section{Ordered algebraic structures.}

The goal of this section is to describe ordered abelian semigroups which can be used to define measure theories.  Here the measure takes values in the ordered abelian semigroup.   For this purpose, it is especially useful to study ordered abelian semigroups which also have the least upper bound property.  For further applications, ordered semirings are even more useful.   We begin by defining a particular ordered semiring, namely $\OB$.

\begin{defn}\label{BLDefn}  Let $\OB$ denote the set $\{0\}\cup(\integers\times (0,\infty])$ where $(0,\infty]\subset\bar \reals$ and where $\bar\reals=[-\infty,\infty]$ denotes the extended real line.  We use the
lexicographical order relation on $\integers\times (0,\infty]$, so that $(i,t)<(j,s)$ either if $i<j$ and $s,t\in (0,\infty]$ or if $i=j$ and $t<s$.   The element $0\in\OB$ is a least element, $0<(i,t)$ for all $(i,t)\in(\integers\times (0,\infty])$.
We make $\OB$ a topological space with the order topology. 
 
 The elements $(i,\infty)$ are called {\it infinities} of $\OB$.  Define a commutative {\it addition} operation 
on $\OB$ by $$(i,t)+(j,u)
=\begin{cases}(i,t+u) & \mbox{if } i=j \\
                     ( i,t) & \mbox{if } i>j 
                                            \end{cases}
$$
$$    (i,t)+0=(i,t),$$
\noindent with the convention $\infty+a=a+\infty=\infty$ for any $a\in (0,\infty]$. Define a commutative {\it multiplication} on $\OB$ by 
$$(i,t)(j,u)=(i+j, tu),$$ $$0(i,t)=(i,t)0=0,$$ with the convention that $a\infty=\infty a=\infty$ for any $a\in (0,\infty]$.  

We say $(i,t)\in (\integers\times (0,\infty])$ is {\it real} if $i=0$.  When we denote an element of $\OB$ by a single symbol $x=(i,t) \in \OB\setminus \{0\}$, we will use $\level(x)=i$ to denote the {\it level of $x$} and $\real(x)=t$ to denote the {\it real part of $x$}, or  {\it residue of $x$}, which lies in $(0,\infty]$.   We make the convention that $\real(0)$ and $\level(0)$ are undefined.
\end{defn}

We observe that the subset of $\OB$ which we identify with the positive extended reals, namely $\{(0,t)\in \OB:t \in (0,\infty]\}=\{x\in \OB:\level(x)=0\}$, has the usual addition and multiplication of the positive extended reals, and also 
has the usual (order) topology.

It is routine to verify the following:

\begin{lemma} $\OB$ is an ordered commutative semiring with multiplicative identity $(0,1)$ and additive identity $0$.  \end{lemma}

Figure \ref{OrderParameter} shows $\OB$ as a topological space with the order topology.

For some applications it is useful to extend $\OB$ slightly.

\begin{defn}  Let $\bar\OB$ denote $\OB\cup\{\Infty\}$ with operations and the order in $\OB$ extended in the obvious way: $\Infty>x$ for all $x\in \OB$.  $\Infty+x=x+\Infty=\Infty$ for all $x\in \OB$.  Finally, $\Infty\cdot x=x\cdot \Infty=\Infty$ unless $x=0$, and $\Infty\cdot 0=0\cdot \Infty=0$.  We make the convention that $\level(\Infty)$ is also undefined, and otherwise $\level(x)$ is defined as for $\OB$.
\end{defn}

To avoid confusion, we define the terms ``ordered abelian semigroup," ``ordered commutative semiring" and ``ordered semifield" as used in this paper.

\begin{defn}  An {\it ordered abelian semigroup} is a set G equipped with a binary operation +, a total ordering $<$ and an element $0$ satisfying the following axioms for any $a,b,c\in G$:

\begin{tightenum}
\item $(a + b) + c = a + (b + c)$
\item $0 + a = a + 0 = a$
\item $a + b = b + a$
\item 0 is the least element and for all $a$ and $b$, $a+b\ge b$.
\setcounter{zahl}{\value{enumi}}
\end{tightenum}
\end{defn}

\begin{defn}  An {\it ordered commutative semiring} is a totally ordered set $R$, with order $<$, equipped with two binary operations addition, +, and multiplication and with elements 0,1, such that $(R,+)$ is an ordered abelian semigroup and such that the following are satisfied for any $a,b,c\in R$:

\begin{tightenum}
\setcounter{enumi}{\value{zahl}}
\item $ab=ba$
\item $(ab)c = a(bc)$
\item $1a = a1 = a$
\item $a(b + c) = (ab) + (ac)$
\item $0a = a0 = 0$
\setcounter{zahl}{\value{enumi}}
\end{tightenum}

An {\it ordered semifield} is a commutative semiring with the additional property that every non-zero element $a$ has a multiplicative inverse $\bar a$ such that:
\begin{tightenum}
\setcounter{enumi}{\value{zahl}}
\item $a\bar a=\bar a a=1$
\end{tightenum}
\end{defn}

\begin{figure}[H]
\centering
\scalebox{0.8}{\includegraphics{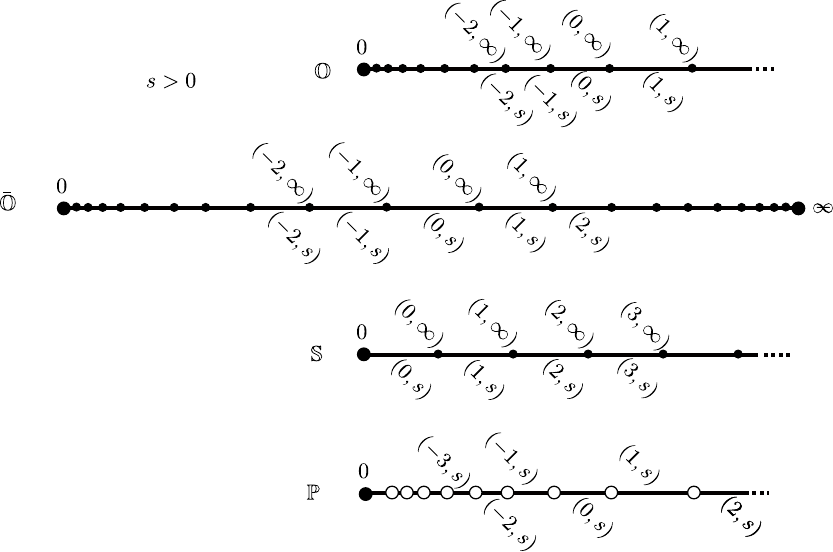}}
\caption{\small The ordered commutative semirings  $\OB$, $\bar\OB$, $\SB$, and $\PB$ viewed as topological spaces.}
\label{OrderParameter}
\end{figure}

Next, we will describe more general constructions for combining {\it ordered algebraic structures}  which can be ordered sets, ordered abelian semigroups,  ordered abelian groups, or ordered commutative semirings, or ordered semifields, to obtain new ordered algebraic structures.

\begin{defn}  Suppose $A$ is an ordered set and suppose $B$ is an ordered abelian semigroup, an ordered commutative semiring, or ordered semifield.  Then we define {\it $A$ insert $B$} as $$ A \ins B=\left[ A\times (B\setminus \{0\})\right]\cup \{0\},$$ which becomes an ordered abelian semigroup with order relation and operations described below.  The zero added to $A\ins B$ is distinct from the zero removed from $B$.  If $A$ is an ordered abelian semigroup and $B$ is an ordered commutative semiring, $A\ins B$ becomes an ordered semiring.  If $A$ is an ordered abelian group and $B$ is an ordered semifield, then $A\ins B$ is an ordered semifield.

We define an order relation $<$ on $A\ins B$ as follows:  

\begin{tightenum}
\item $(g,s)<(h,t)$ if either if $g<h$ or if $g=h$ and $s<t$.
\item $0<(g,s)$ for all $(g,s)$.
\end{tightenum}

Now we define the addition operation on $A\ins B$.  
The commutative {\it addition} operation 
on $A\ins B$ is given by $$(g,s)+(h,t)
=\begin{cases}( g,s) & \mbox{if } g>h \\
                     (g,s+t) & \mbox{if } g=h
                                            \end{cases}
$$
$$0+(g,s)=(g,s)+0=(g,s)$$

In case $B$ is an ordered abelian semigroup, this defines the ordered abelian semigroup $A\ins B$.

Provided $A$ is an ordered abelian semigroup and $B$ is an ordered commutative semiring, we define
 a commutative {\it multiplication} on $A\ins B$ by 
$$(g,s)(h,t)=(g+h, st),$$ $$0(g,s)=(g,s)0=0,$$
to obtain an ordered semiring $A\ins B$.  If $B$ has a multiplicative identity $1$, then $A\ins B$ has a multiplicative identity $(0,1)$.

If $A$ is an ordered abelian group and $B$ is an ordered semifield, then $A\ins B$ is an ordered semifield.  In this case, the multiplicative inverse of $(g,s)$ is $(-g,\bar s)$, where $\bar s$ is the multiplicative inverse of $s$ in $B$.  

In cases where $A\ins B$ is either an ordered abelian semigroup or an ordered commutative semiring we can extend $A\ins B$ by including an infinity, $\Infty$, then we define $A\ens B$ as $A\ins B\cup \{\Infty\}$ and extend the operations and order relation as follows:

For all $x\in A\ens B$, $\Infty>x$, $\Infty+x=x+\Infty=\Infty$.  Finally, for semirings $\Infty\cdot x=x\cdot \Infty=\Infty$ unless $x=0$, and $\Infty\cdot 0=0\cdot \Infty=0$.  The operation $\ens$ is called {\it extended insertion}.

When we denote a non-zero element of $A\ens B$ by a single symbol $x=(g,s)$, we will use $\level(x)=g$ to denote the {\it level of $x$} and $\real(x)=s$ to denote the {\it residue of $x$}, which lies in $B \setminus\{0\}$.   We make the convention that $\real(0)$ and $\level(0)$ is undefined; similarly $\level(\Infty)$, and $\residue(\Infty)$ are undefined.  

In some calculations it is useful to interpret the pair $(g,0))\in A\times B$ as $0\in A\ens B$.

 \end{defn}
 
 We summarize properties of the insertion operation in the following lemma.

\begin{lemma}\label{WellDefinedLemma}  (a) If $A$ is an ordered set and $B$ is an ordered abelian semigroup
then $A\ins B$ and $A\ens B$  are ordered abelian semigroups. 

\noindent (b) If $A$ is an ordered abelian group or ordered abelian semigroup and $B$ is an ordered commutative semiring 
then $A\ins B$ and $A\ens B$ are ordered commutative semirings.    In this case, $\level(xy)=\level(x)+\level(y)$. If $x\ne 0$, $y\ne 0$, then the level function $\level$ in $A\ins B$ satisfies $\level(x+y)=max(\level(x),\level(y))$.

\noindent (c)  If $A$ is an ordered abelian group and $B$ is an ordered semifield
then $A\ins B$ is an ordered semifield. 
\end{lemma}

\begin{remark}   We describe the operation $\ins$ as an ``insertion."  To construct $A\ins B$, we insert a copy of $B\setminus \{0\}$ at every element of $A$, with its order, then add a different $0$ to construct $A\ins B$.   To construct $A\ens B$ we also add a largest element $\Infty$.
 \end{remark}

\begin{exs}
\begin{intenum}
\item  Let $\naturals_0=\naturals\cup \{0\}=\{0,1,2,3,\ldots\}$ and let $\SB$ denote $\naturals_0\ins [0,\infty]$, constructed using the commutative semiring $[0,\infty]\subset \bar\reals$.  Then $\SB$ is an ordered commutative semiring, which is shown as a topological space in Figure \ref{OrderParameter}.   
\item $\integers\ins [0,\infty]$ is the commutative semiring we described above as $\OB$.  Both $\OB$ and $\bar\OB=\integers\ens [0,\infty]$ are shown in Figure \ref{OrderParameter}.   
\item  Since, for example $\integers \ins[0,\infty]$ can be regarded as an ordered abelian semigroup, we can construct a new ordered commutative semiring $(\integers\ins [0,\infty])\ins[0,\infty]$.   This semiring cannot so easily be pictured,

\end{intenum}
\end{exs}

Obviously we can construct examples by iterating $\ins$ or $\ens$ operations as often as we want, with brackets to indicate the order of operations.  Most of these examples do not seem particularly useful.

For applications to measure theory, we will need to define countably infinite sums, at least those sums which arise as sums of measures of a countable collection of disjoint measurable sets.

\begin{defn}  In $\OB$ we can define countably infinite sums of elements $x_n\in \OB$, provided the summands have uniformly bounded levels.  In this case, assuming the maximum of $\level(x_n)$ is $M$, it is natural to define
$\displaystyle \sum_{n=1}^\infty x_n:=\sum_{\level(x_n)=M}x_n$, which makes sense since $\displaystyle \sum_{\level(x_n)=M}x_n=(M,\sum_{\level(x_n)=M}\real(x_n))$, which is expressed in terms of a countable sum of positive extended real numbers.

If levels $\level(x_n)$ are not uniformly bounded, the countable sum does not make sense in $\OB$.   However, in $\bar \OB$, we can, as before, define $\displaystyle \sum_{n=1}^\infty x_n =\sum_{\level(x_n)=M}x_n$ if $\{\level(x_n)\}$ has maximum value $M$.  In case $\{\level(x_n)\}$ is unbounded above, we define $\displaystyle \sum_{n=1}^\infty x_n=\Infty$.
\end{defn}

The above definition can be summarized by saying that $\displaystyle \sum_{n=1}^\infty x_n$ is the least upper bound of partial sums whenever the least upper bound exists.

We shall see that we can evaluate countable sums in many other ordered algebraic structures.

\begin{defn}  An ordered algebraic structure $A$ has the {\it least upper bound property (lub property)} if every set $S\subset A$ which is bounded above has a least upper bound $\sup(S)$.  
\end{defn}

\begin{proposition} \label{LUBProp}Suppose $A$ is an ordered set with the property that every bounded non-empty set has a greatest element.  Suppose $B$ is an ordered abelian semigroup.  Suppose $B$ has the least upper bound property and either:
\begin{tightenum}
\item $B$ has a greatest element $\infty$, or
\item $B$ has a least element $p>0$,  and $A$ has the property that every non-empty set which is bounded below has a least element.
\end{tightenum}

\noindent Then $A\ins B$ have the least upper bound property, and also $A\ens B$ has the least upper bound property.
\end{proposition}

\begin{proof}
Suppose $A$ and $B$ are as in the statement.  We will show that $A\ins B$ has the lub property.  Suppose $S\subset A\ins B$ is bounded above.  If $S=\{0\}$, then the least upper bound is $0$, so we may assume $S$ contains elements other than 0.  Then $$T=\{a\in A: \text{ there exists } b\in B \text{ such that } (a,b)\in S\}$$ is bounded above and has a greatest element $M$.  Let $U=\{b\in B: (M,b)\in S\}\ne\emptyset$. 

\noindent {\it Case (i).} If $B$ has a greatest element $\infty$, since $B$ has the lub property and $U$ is bounded above by $\infty$, $U$ has an lub  $N$ say.  Then $(M,N)$ is the lub for $S$.  

\noindent{\it Case(ii)}  If $U$ is bounded above, again let $N$ be the least upper bound of $U$.   As before $(M,N)$ is the lub.   

Now we assume $U$ is not bounded above in $B$.  We first consider the subcase that $M$ is not the greatest element of $A$ (which is guaranteed if $A$ has no greatest element).   Then let $M'$ be the least element of $A$ greater than $M$.  Clearly then $(M',p)$ is the least upper bound of $S$.  In the remaining case, $M$ is the greatest element of $A$.   Then $U$ is bounded above if and only if $S$ is bounded above.   Thus $U$ is bounded above, and we are in a case we already considered:  letting $N$ be the lub of $U$, $(M,N)$ is the lub of $S$.

We must also show that $A\ens B$ has the least upper bound property.  Suppose $S\subset A\ens B$.   If $\Infty\in S$, then $\Infty$ is the lub of $S$.   Otherwise $S\subset A\ins B\subset A\ens B$.   If $S$ is bounded above by an element of $A\ins B$, we have already shown it has a least upper bound.  Otherwise, the only upper bound of $S$ is $\Infty$, which shows that $\Infty$ is the lub of $S$. 
\end{proof}

\begin{defn}  Suppose $\LB$ is any ordered abelian semigroup with the lub property.   Suppose $x_n\in \LB$ and the partial sums of $\displaystyle \sum_{n=1}^\infty x_n$ are bounded above.   Then $\displaystyle \sum_{n=1}^\infty x_n$ is defined to be the lub of the partial sums.
\end{defn}

\begin{proposition} \label{SummableProp}  Suppose $A$ is an ordered set with the property that every non-empty set which is bounded above has a greatest element.  Suppose $B$ has the lub property.  Suppose $x_i\in \LB$ where $\LB=A\ins B$ or $\LB=A\ens B$.  Let $S_n$ be the $n$-th partial sum of $\displaystyle \sum_{i=1}^\infty x_i$ and suppose $\{S_n\}$ is uniformly bounded.  If $\{\level(S_n)\}$ is uniformly bounded, let $M$ be the largest element, then $\displaystyle \sum_{i=1}^\infty x_i=\left(M,\sum_{\level(x_i)=M}\real(x_i)\right)$.   Otherwise $\displaystyle \sum_{i=1}^\infty x_i=\Infty$.
\end{proposition}

\begin{proof}  If $\{\level(S_n)\}$ is uniformly bounded, with largest element $M$, then $\displaystyle \sum_{i=1}^\infty x_i=\lub\{S_n\}=\lub\{S_n:\level(S_n)=M\}=(M,\lub
\{\real(S_n)\})=\left(M,\sum_{i=1}^\infty \real(x_i)\right)$.  Otherwise  $\lub\{S_n\}=\Infty$.
\end{proof}

\begin{corollary}\label{NNCor}   The ordered abelian semigroup $\SB=\naturals_0\ins [0,\infty]$ has the lub property.
\end{corollary}

\begin{proof}  To prove the lub property observe that $[0,\infty]$ has a greatest element $\infty$, so we can apply Lemma \ref{LUBProp} in case (i).   \end{proof}

More generally, we have:

\begin{lemma} If $[0,\infty]$ denotes the interval in $\bar\reals$, then 
$$\bar\SB_r=(\naturals_0\ens(\naturals_0\ens\cdots\ens(\naturals_0\ens(\naturals_0\ens [0,\infty]))\cdots)),$$ with $r$ insertions in $\naturals_0$, is an ordered abelian semiring with the lub property, and a greatest element. 
\hop

\noindent Similarly $$\bar\OB_r=(\integers\ens(\integers\ens\cdots\ens(\integers\ens(\integers\ens [0,\infty]))\cdots)),$$ with $n$ insertions in $\integers$, is an ordered abelian semiring with the lub property  and a greatest element.
\end{lemma}

\begin{proof}  We can prove that $\bar \SB_r$ has the lub property using induction starting with $\naturals_0\ens[0,\infty]$.  Clearly $[0,\infty]$ has a greatest element and has the lub property.  Now we inductively apply $\naturals_0\ens$ to the previous result, and at every step of the induction apply Lemma \ref{LUBProp} (i) using $A=\naturals_0$.  

Essentially the same proof works for the second statement.
\end{proof}

\begin{proposition}  $\displaystyle (\naturals_0\ins(\naturals_0\ins\cdots\ins(\naturals_0\ins(\naturals_0\ins \naturals_0))\cdots))$ has the lub property and a least positive element. $\displaystyle (\naturals_0\ens(\naturals_0\ens\cdots\ens(\naturals_0\ens(\naturals_0\ens \naturals_0))\cdots))$ has the lub property, a least positive element, and a greatest element.
\end{proposition}
\begin{proof}  For the first statement, we apply Lemma \ref{LUBProp} (ii)  inductively to prove the lub property.   For the second statement, we use  Lemma \ref{LUBProp} (i) inductively. In this case, the fact there is a greatest element in the second ordered abelian group is obvious.
\end{proof}

\begin{proposition} $\PB=\integers\ins [0,\infty)$ is an ordered semifield.    Also,
$$\PB_r=(\integers\ins(\integers\ins\cdots\ins(\integers\ins(\integers\ins [0,\infty)))\cdots)),$$ with $n$ insertions in $\integers$ is an ordered semifield.  None of these have the lub property.  
\end{proposition}

\begin{proof} We prove this by induction starting from the fact that $[0,\infty)$ is an ordered semifield.   If we know that $\PB_{r-1}$ is an ordered semifield
it follows by Lemma \ref{WellDefinedLemma}(c) that $\PB_r=\integers\ins \PB_{r-1}$ is an ordered semifield.
\end{proof}

$\PB$ is a sub-semiring of $\OB$, see Figure \ref{OrderParameter}.  
For fixed $i\in \integers$ the set $S=\{(i,t):0<t<\infty\}$ is bounded above by $(i+1,1)$, but it does not have a least upper bound.  So $\PB$ certainly does not have the least upper bound property.   We use the symbol $\PB$ for this ordered semifield because it can be used for probability calculations, see Section \ref{Probability}.  As we observed in the introduction, $\PB$ also plays a role in the description of transverse measures for a lamination.  

The following gives another large class of ordered abelian semigroups with the lub property, which can be used for measure theories, possibly useful for measures on large cardinality sets.  This is an immediate corollary of Proposition \ref{LUBProp}(i).

\begin{corollary}\label{WellOrderCor} Suppose $A$ is a well-ordered set and $B$ is an ordered abelian semigroup with the lub property.  We reverse the order on $A$, so it has the property that any non-empty subset has a greatest element.  Suppose also that $B$ has a greatest element $\infty$.
Then $A\ins B$  has the lub property and a greatest element.
\end{corollary}

Notice that in Corollary \ref{WellOrderCor}, since the ordered abelian semigroup $A\ins B$ has a greatest element, nothing is gained by considering $A\ens B$.

\begin{proposition}  Suppose $A_1,A_2,\ldots, A_r$ are well-ordered sets, each with its ordering reversed.   Then the ordered abelian semigroup $A_1\ins( A_2\ins\cdots (A_r\ins [0,\infty])\cdots))$ has the lub property.

\end{proposition}

Here is another construction for ordered abelian semigroups with the lub property.

\begin{lemma}  Suppose $A$ is an ordered set with the lub property and $B$ is an ordered abelian semigroup with the lub property and a least element $p>0$.   Then $A\ins B$ and $A\ens B$ have the lub property.
\end{lemma}

\begin{proof}  We first prove the statement for $A\ins B$.   Suppose $S$ is a non-empty set in $A\ins B$.   If $S=\{0\}$, then the least upper bound is $0$, so we may assume $S$ contains elements other than 0.  Let $M=\lub \{\level(x):x\in S\}$.  Let  $U=\{b\in B: (M,b)\in S\}$.  If $U$ is non-empty, let $N=\lub( U)$, then $\lub (S)=(M,N)$.  If $U$ is empty, then $\lub (S)=(M,p)$.   It remains to consider $A\ens B$.    If $\Infty\in S$, then $\Infty$ is the lub of $S$.   Otherwise $S\subset A\ins B\subset A\ens B$.   If $S$ is bounded above by an element of $A\ins B$, we have already shown it has a least upper bound.  Otherwise, the only upper bound of $S$ is $\Infty$, which shows that $\Infty$ is the lub of $S$. 
\end{proof}

\begin{proposition} The ordered abelian semirings $\TB_r=(\cdots(([0,\infty]\ins\naturals_0)\ins\naturals_0)\cdots\ins\naturals_0)$ (with $\naturals_0$ inserted $r$ times) have the lub property.   Also $\bar\TB_r=(\cdots(([0,\infty]\ens\naturals_0)\ens\naturals_0)\cdots\ens\naturals_0)$ has the lub property.
\end{proposition}

\begin{proof}  Apply the above lemma inductively.
\end{proof}

\hop

It is natural to ask whether in some sense the $\ins$ operation is associative.  The following example shows that is not the case.

\begin{example}\label{AssociativeExample}  Consider the two ordered semirings $\naturals_0\ins(\naturals_0\ins\naturals_0)$ and $(\naturals_0\ins\naturals_0)\ins\naturals_0$.   One might hope that $\psi((i,(j,k)))=((i,j),k)$, $\psi(0)=0$ defines an isomorphism of ordered semirings, $\psi:\naturals_0\ins(\naturals_0\ins\naturals_0)\to (\naturals_0\ins\naturals_0)\ins\naturals_0$.    That is not the case.    For example in the first group $\naturals_0\ins(\naturals_0\ins\naturals_0)$ we have $(1,(1,1))\cdot (2,(1,1))=(3,(2,1))$ whereas in the second group $(\naturals_0\ins\naturals_0)\ins\naturals_0$ we have 
$((1,1),1)\cdot ((2,1),1)=((2,1),1)$.  In fact, $\psi$ is not even well-defined, since for example, $\psi(0,(0,1))=((0,0),1)$ and $(0,0)\notin \naturals_0\ins\naturals_0$, so $((0,0),1)\notin (\naturals_0\ins\naturals_0)\ins\naturals_0$.
\end{example}

We need a notation for the elements of $\bar\OB_r$ and $\PB_r$ in order to prove some results which we will need later in the paper.  

\begin{defn}\label{ONotation}  We write the elements of $\bar\OB_r$ in the form \\
\centerline{$\bar\OB_r=\left\{(i_1,i_2,\ldots, i_s,t):s\le r,\ t\in (0,\infty] \text{ if } s=r, \ t=\infty \text{ if } s<r\right\}.$}
We say the {\it level} of an element $x=(i_1,i_2,\ldots, i_s,t)$ is $\level(x)=(i_1,i_2,\ldots, i_s)$, the {\it residue} is $\residue(x)=t$.  The elements of the form $(i_1,i_2,\ldots, i_s,\infty)$ for $0\le s\le r$ are called {\it infinities}.  
\end{defn}

When $s=0$, with this notation the level of $(\infty)$, the largest element of $\bar\OB_r$ is the empty sequence of integers, so in a sense this agrees with our previous convention that the level is undefined.

When we use the notation of Definition \ref{ONotation} for $\PB_r\subset\bar\OB_r$, the infinities are absent and we get\\
\centerline{$\PB_r=\left\{0\right\}\cup\left\{(i_1,i_2,\ldots, i_r,t):\ t\in (0,\infty)\right\}.$}

Note that $(i_1,i_2,\ldots, i_s,\infty)$ is ``the infinity that appears after $r-s$ insertions."  More precisely, if $s<r$ and if we write $\bar\OB_r=\integers\ens(\integers\ens(\integers\cdots\ens(\integers\ens\bar\OB_{r-s})\cdots))$, with $s$ $\integers$'s, then $(i_1,i_2,\ldots, i_s,\infty)$ is the infinity for the $\bar\OB_{r-s}$ inserted at level $(i_1,i_2,\ldots, i_s)$ .  In particular, if $s=0$, $(\infty)$ is the largest element in $\bar\OB_r$.   When we use this notation, we no longer need to use the notation $\Infty$.  The infinities $(i_1,i_2,\ldots,i_s,\infty)$ are ordered lexicographically in $\bar\OB_r$ but not quite as in a dictionary:  a shorter string with the initial letters of a longer string is larger.   Thus $(2,1,3,\infty)>(2,1,3,5,\infty)$.   If we order levels in the same way, skipping the last entry, then the level function $\level$ preserves order.  We want a better description of the addition and multiplication operations in $\bar\OB_r$.   Unfortunately, these descriptions are awkward, involving some further definitions as follows.

\begin{defn}  The set of levels $(i_1,i_2, \ldots, i_s)$, $0\le s\le r$ for elements of $\bar\OB_r$ is ordered as follows.  When we compare $(i_1,\ldots,i_p)$ and $(j_1,\ldots, j_q)$ we let $k=\min\{p,q\}$.   Then
$(i_1,\ldots,i_p)<(j_1,\ldots, j_q)$ if  $(i_1,\ldots,i_k)<(j_1,\ldots, j_k)$ in the lexicographical ordering or if $q<p$ and $i_1=j_1,\ i_2=j_2,\ \ldots,\  i_q=j_q$.

We also define an addition operation on the set of levels of $\bar\OB_r$.  Namely, $(i_1,\ldots,i_p)+(j_1,\ldots, j_q)=(i_1+j_1,i_2+j_2,\ldots,i_k+j_k)$.
\end{defn}

We observe that the order relation we defined above on levels of $\bar\OB_r$ is just the order relation from the ordered abelian semigroup  $\integers\ens(\integers\ens\cdots\ens(\integers\ens\{0\})\cdots)$, with $r$ copies of $\integers$, where $\{0\}$ is the trivial semigroup.   The elements of  $\integers\ens(\integers\ens\cdots\ens(\integers\ens\{0\})\cdots)$ have the form $(i_1,i_2,\ldots,i_s, \infty)$ where $s\le r$.   Ignoring the last entry, we get the levels of $\bar\OB_r$.  The addition operation on levels of $\bar\OB_r$ {\it does not} correspond to the addition operation in $\integers\ens\integers\ens\cdots\ens\integers\ens\{0\}$.

\begin{proposition} Writing an arbitrary element of $x\in\bar\OB_r$  as $x=(\level(x),\residue(x))$.  

 \noindent (a) If $x\ne 0$ and $y\ne (\infty)$, the order relation on $\bar\OB_r$ can be expressed as 

 $$x<y \mbox{ if } \level(x)<\level(y)\mbox{ or } \level(x)=\level(y) \mbox{ and } \residue(x)<\residue(y).  
$$

\noindent Also $0< y$ for any $y\ne 0$, and $x <(\infty)$ for any $x\ne(\infty)$.

\noindent (b) If $x\ne 0,(\infty)$ and $y\ne 0,(\infty)$, the $+$ operation in  $\bar\OB_r$ can be expressed as:
 
  $$x+y
=\begin{cases} \max(x,y) & \mbox{if } \level(x)\ne \level(y) \\
                     (\level(x),\residue(x)+\residue(y)) & \mbox{if } \level(x)= \level(y)
                                            \end{cases}.
$$

\noindent Also, if $x=0$, $x+y=y$; and if $y=(\infty)$, $x+y=(\infty)$.

\noindent (c)  If $x\ne 0,(\infty)$ and $y\ne 0,(\infty)$, the product can be defined by
 $$xy
=\left(\level(x)+\level(y), \residue(x)\residue(y)\right).
$$

\noindent  Also, if $x=0$ or $y=0$, $xy=0$; and if $x\ne 0$ and $y=(\infty)$, the $xy=(\infty)$.
\end{proposition}

\begin{proof}  We use induction on $r$ in each part.

 (a) For $\bar\OB=\bar\OB_1$ the statement is true by the definition of the order in $\bar\OB= \integers\ens[0,\infty]$.  This begins our induction.   Suppose the statement is true for $\bar\OB_{r-1}$.   Now consider $\bar\OB_r=\integers\ens\bar\OB_{r-1}$, and let $x,y\in \bar\OB_r$.  Initially, we assume neither $x$ nor $y$ is $0$ or $(\infty)$, so we assume $x=(i_1,\hat x)\in\integers\ens\bar\OB_{r-1}$ and $y=(j_1,\hat y)$.  Then we may assume $x=(i_1,(i_2,i_3,\ldots, i_p,t))$ with $p\le r$, and $y=(j_1,(j_2,j_3,\ldots j_q,u))$ with $q\le r$.  Using the definition of the order in $\bar\OB_r=\integers\ens\bar\OB_{r-1}$, we have $x<y$ if $i_1<j_1$ or $i_1=j_1$ and $\hat x<\hat y$.  
 
 If $i_1<j_1$, then $(i_1,i_2,\ldots, i_p)<(j_1,j_2,\ldots, j_q)$, so $\level(x)<\level(y)$.  
 
 If $i_1=j_1$, then $x<y$ if $\hat x<\hat y$.   But, by our induction hypothesis, we know how to check whether $\hat x<\hat y$.  We know $\hat x<\hat y$ means $(i_2,\ldots, i_p,t)<(j_2,\ldots, j_q,u)$, which is true if $\level(\hat x)<\level(\hat y)$ or $\level(\hat x)=\level(\hat y)$ and $t<u$.  The condition $\level(\hat x)<\level(\hat y)$ is true if $q<p$ or $(i_2,\ldots, i_k)<(j_2,\ldots, j_k)$, where $k=\min(p,q)$.   This is the same as saying  $(i_1,i_2,\ldots, i_p)<(j_1,j_2,\ldots, j_q)$  (assuming $i_1=j_1$). Hence, we have $\level(x)<\level(y)$ if $\level(\hat x)<\level(\hat y)$.   If $\level(\hat x)=\level(\hat y)$, then clearly $\level(x)=\level(y)$ and $t<u$.   In other words $\level(x)=\level(y)$ and $\residue(x)<\residue(y)$.
 
 So far, we have not mentioned the possibility that $x=0$ or $y=(\infty)$.  Since $0$ is the smallest element in $\bar\OB_r$ and $(\infty)$ is the largest, the ordering is clear in these cases.
 
 (b)  The induction for the second statement, involving the addition operation is similarly tedious.   If $x$ and $y$ belong to $\bar\OB=\bar\OB_1$, the statement is true by the definition of the addition operation in $\bar\OB=\integers\ens [0,\infty]$.  Suppose the statement is true for $\bar\OB_{r-1}$.   We assume first that $x\ne 0$, $y\ne 0$, $x\ne (\infty)$, $y\ne (\infty)$.  Again let $x=(i_1,\hat x)\in\integers\ens\bar\OB_{r-1}$ and $y=(j_1,\hat y)$.  Then by the definition of the addition operation in $\integers\ens\bar\OB_{r-1}$, we have 
 
 $$x+y= \begin{cases}(i_1,\hat x) & \mbox{if } i_1>j_1\\
                     ( i_1,\hat x+\hat y) & \mbox{if } i_1=j_1.
                                            \end{cases}
$$

\noindent  In the first case, $i_1>j_1$, so $\level(x)>\level(y)$, so the above formula agrees with the statement.  In the second case, $i_1=j_1$, if $\level(\hat x)>\level(\hat y)$, we conclude $\level(x)>\level(y)$ and the formula gives $x+y=(i_1,\hat x)=x$.  In the second case, if $\level(\hat x)=\level(\hat y)$, we conclude $\level(x)=\level(y)$, and $x+y=(i_1,(\level(\hat x), \residue(\hat x)+\residue(\hat y)))=(\level(x),\residue(x)+\residue(y))$, so again we get the desired answer.  

It is easy to verify the statement in the remaining cases, where $x=0$, $y= 0$, $x=(\infty)$, or $y= (\infty)$. 

(c)  Again, the interpretation of the multiplication operation in the statement follows immediately from the definition of multiplication in $\bar\OB_1=\bar\OB=\integers\ens [0,\infty]$.  This starts the induction.   Suppose the statement is true for $\bar\OB_{r-1}$.   We assume first that $x\ne 0$, $y\ne 0$, $x\ne (\infty)$, $y\ne (\infty)$.  Again let $x=(i_1,\hat x)\in\integers\ens\bar\OB_{r-1}$ and $y=(j_1,\hat y)$.  Then by the definition of the multiplication operation in $\integers\ens\bar\OB_{r-1}$, we have 
 
 $$xy=(i_1 +j_1,\hat x\hat y).$$
 
 \noindent By the induction hypothesis, $\hat x \hat y=(\level(x)+\level(y),\residue(\hat x)\residue(\hat y))$, so 
 
 $$xy=(i_1 +j_1,\level(\hat x)+\level(\hat y),\residue(\hat x)\residue(\hat y)).$$

The addition of levels was defined so that this yields 

 $$xy=(\level(x)+\level(y),\residue( x)\residue(y)).$$
\noindent 
It is easy to verify the statement in the remaining cases, where $x=0$, $y= 0$, $x=(\infty)$, or $y= (\infty)$.

\end{proof}

\begin{defn}  If $\KB$, $\LB$  are ordered abelian semigroups, an {\it ordered abelian semigroup homomorphism from $\KB$ to $\LB$} is a map $\psi:\KB\to \LB$ satisfying:  (i) $\psi(0)=0$,   (ii)  $\psi(x+y)=\psi(x)+\psi(y)$ and (iii)  if $x<y$ then $\psi(x)<\psi(y)$.  An {\it ordered abelian semigroup automorphism of $\LB$} is an an ordered abelian semigroup homomorphism from $\psi:\LB \to\LB$ which has an inverse homomorphism.  We will denote the group of ordered abelian semigroup automorpisms of $\LB$ by $\Aut(\LB)$.
\end{defn}

Every positive scalar $\lambda\in (0,\infty)$ yields an ordered abelian semigroup automorphism $\psi_\lambda:[0,\infty]\to[0,\infty]$, $\psi_\lambda(x)=\lambda x$.   In fact, $\reals_+=(0,\infty)$ with the multiplication operation is isomorphic to the group of ordered abelian semigroup automorphisms.   The following proposition gives the analogous statement for $\bar\OB_r$.   We emphasize the automorphisms $\bar\OB_r$ we are discussing here {\it do not respect the multiplication operation} in $\bar\OB_r$.  

\begin{proposition}\label{OrderAutoProp} For every $\lambda\in \PB_r$ there is an ordered abelian semigroup automorphism $\psi_\lambda:\bar\OB_r\to \bar\OB_r$ and the mapping $\lambda \mapsto \psi_\lambda$ is a group isomorphism from $(\PB_r,\cdot)\to \Aut(\bar\OB_r)$.   \end{proposition}

Thus $\PB_r$ can be identified with the group of ordered abelian semigroup automorphisms of $\bar\OB_r$.   

\begin{proof}  First we verify that $\psi_\lambda$ is an automorphism.   Obviously $\psi_\lambda(0)=0$.   Next we check $\psi_\lambda(x+y)=\lambda(x+y)=\lambda x+\lambda y=\psi_\lambda(x)+\psi_\lambda(y)$ using the distributive law for the multiplication in $\bar\OB_r$.   Finally we verify that $\psi_\lambda$ respects the order.   If $x<y$ and $x\ne 0$, $y\ne (\infty)$, then $\level(x)<\level(y)\mbox{ or } \level(x)=\level(y) \mbox{ and } \residue(x)<\residue(y)$.  We compare levels and residues of $\psi_\lambda(x)=\left(\level(\lambda)+\level(x), \residue(\lambda)\residue(x)\right)$ and $\psi_\lambda(y)=\left(\level(\lambda)+\level(y), \residue(\lambda)\residue(y)\right)$.   If the levels of $x$ and $y$ are equal and $\residue(x)<\residue(y)$, then $x$ is not an infinity, then the levels of $\psi_\lambda(x)$ and $\psi_\lambda(y)$ are equal and $\residue(\psi_\lambda(x))= \residue(\lambda)\residue(x)<\residue(\lambda)\residue(y)=\residue(\psi_\lambda(y))$, because $\residue(x)$ cannot be an $\infty$.  Assuming $\level(x)<\level(y)$, if we show $\level(\lambda)+\level(x)<\level(\lambda)+\level(y)$, that is enough to show  $\psi_\lambda(x)< \psi_\lambda(y)$.   Since $\lambda\in \PB_r$, it has the form $(m_1,\ldots, m_r,t)$, where $0<t<\infty$, whereas $\level(x)$ has the form $(i_1,\ldots,i_p)$ with $p\le r$, while $\level(y)$ has the form  $(j_1,\ldots,j_q)$, with $q\le r$.  So $\level(\lambda)+\level(x)=(m_1+i_1,\ldots, m_p+i_p)<\level(\lambda)+\level(y)=(m_1+j_1,\ldots, m_q+j_q)$.

Special cases where $x=0$ and/or $y=(\infty)$ are clear.  
\end{proof}

\hop\hop
\noindent  {\it Mixed insertions for ordered abelian semigroups.}
\hop
We can define more general insertion operations yielding ordered abelian semigroups.  Namely, we can insert different semigroups at different levels.   The reader may skip this discussion, but the resulting ordered algebraic structures may be useful in some applications, e.g. describing laminations with particular kinds of transverse measures with values in ordered abelian semigroups obtained by mixed insertions, see Example \ref{MixedExample}.

\begin{defn}  Suppose $A$ is an ordered set.  For each $a\in A$, let $B_a$ be an ordered abelian semigroup.   Then we define 
$$A\bigins_{a\in A}B_a=\left\{ (a,b): a\in A, b\in (B_a\setminus\{0\})\right\}\cup\{0\}$$
\noindent We make this set into an ordered abelian semigroup by defining the order relation and the addition operation in the usual way:

\begin{tightenum}
\item $(g,s)<(h,t)$ if either if $g<h$ or if $g=h$ and $s<t$.
\item $0<(g,s)$ for all $(g,s)$.
\end{tightenum}
The addition operation is commutative and given by 
 $$(g,s)+(h,t)
=\begin{cases}( g,s) & \mbox{if } g>h \\
                     (g,s+t) & \mbox{if } g=h
                                            \end{cases}
$$
$$0+(g,s)=(g,s)+0=(g,s)$$
If we wish to extend $\displaystyle A\bigins_{a\in A}B_a$ by including an infinity, $\Infty$, then we define 

$$A\bar{\bigins_{a\in A}}B_a=A\bigins_{a\in A}B_a\cup \{\Infty\}$$
and extend the order relation and addition operation as before, so that for all $x\in A\bar\bigins_{a\in A}B_a$, $\Infty>x$, $\Infty+x=x+\Infty=\Infty$.  

As before, when we denote a non-zero element of $\displaystyle A\bar\bigins_{a\in A}B_a$ by a single symbol $x=(g,s)$, we will use $\level(x)=g$ to denote the {\it level of $x$} and $\real(x)=s$ to denote the {\it residue of $x$}, which lies in $B_g\setminus\{0\}$.   We make the convention that $\real(0)=0$ and $\level(0)$ is undefined (and $\level(\Infty)$ is undefined if $\Infty$ is included). 

In the same way, we define $$A\bigins_{k\le a\le h}B_a=\left\{ (a,b): a\in A, k\le a\le h, b\in (B_a\setminus\{0\})\right\}\cup\{0\}$$ 
which we will also write $\displaystyle A\bigins_{a=k}^hB_a$.  We can similarly define  $\displaystyle  A\bigins_{a\le h}B_a$ and $\displaystyle A\bigins_{k\le a}B_a$.
\end{defn}

\section{Measures.}

In this section we will assume that $\LB$ is some ordered abelian semigroup.   Often $\LB$ will have the least upper bound property.  
\begin{defn}\label{LMeasureDef}  Suppose $\LB$ is an ordered abelian semigroup.  Let $(X,\Sigma)$ be a measurable space with $\sigma$-algebra $\Sigma$.   An {\it $\LB$-measure} $\mu$  assigns an element $\mu(E)$ of $\LB$ to each measurable set $E$ such that the following {\it measure axioms} hold :

(i) $\mu(\emptyset)=0$, 

(ii) If $\{E_k\}_{k\in I}$ is a countable collection of disjoint measurable sets in $X$, then 
$$\mu\left(\bigcup_{k\in I} E_k\right )=\sum_{k\in I} \mu(E_k)\in \LB.$$

\end{defn}

Typically, $\LB$ will have the lub property, so that the existence of the sum in (ii) is immediate.   But, for example, it is possible to define $\PB$-measures although $\PB$ does not have the lub property, see Section \ref{Probability}, an appendix.
For measure theories involving measures with values in $\reals$, one allows infinite measures although $\infty\notin \reals$.  Using our ordered algebraic structures instead of $\reals$, we are often forced to include infinities (at different levels) so that in condition (ii) of the above definition $\sum_{k\in I} \mu(E_k)$ lies in $\LB$.   If $\LB$ has the least upper bound property and has a greatest element, then we can be sure that  $\sum_{k\in I} \mu(E_k)\in \LB$ always lies in $\LB$.  However, even if $\LB$ does not have a greatest element or the lub property, it is sometimes possible to construct $\mu$ so that $\sum_{k\in I} \mu(E_k)$ is always in $\LB$.

Let us consider now the special case of an $\SB$-measure or $\bar\SB$-measures.  Since $\SB$-measures are $\bar\SB$-measures, we may as well work with $\bar\SB$-measures, specifying $\SB$-measures only when necessary.    Recall $\bar\SB=\naturals_0\ens [0,\infty]$.  The definitions and ideas developed for $\bar\SB$-measures apply equally to $\bar\OB$-measures and even more generally.  

\begin{example}  Here is an uninteresting example of a $\bar\SB$-measure, which at least shows that $\bar \SB$-measures exist.   Let $X=[0,\infty)\subset \reals$ be the measurable space with the usual Lebesgue measurable sets.  Let $X_i=[i-1,i)$, $i\in \naturals$.   Let $\nu$ be the Lebesgue measure on $X$.   We define a $\bar\SB$-measure $\mu$ as follows.   If $E$ is measurable, and there is a maximum $i$, say $i=n$, such that $\nu(E\cap X_i)>0$, then we define $\mu(E)=(n,\nu(E\cap X_n))$.   If $\nu(E\cap X_i)=0$ for all $i$, we define $\mu(E)=0$.  If $\nu(E\cap X_i)>0$ for infinitely many $i>0$, we define $\mu(E)=\Infty$.   

The measure axioms for $\mu$ are very easy to check  as follows.   Suppose $\{E_k:k\in I\}$ is a countable collection of disjoint measurable sets.   
If $\level(\mu\left(\bigcup_{k\in I} E_k\right ))=n$, then $\nu((\bigcup_{k\in I}E_k)\cap X_n)>0$ and $\nu((\bigcup_{k\in I}E_k)\cap X_i)=0$ for $i>n$.   By the measure axioms for $\nu$, $\nu((\bigcup_{k\in I}E_k)\cap X_n)=\sum_{k\in I} \nu(E_k\cap X_n)$.  This means that there exist values of $k$ such that $ \nu(E_k\cap X_n)>0$, while for all $k$, $\nu(E_k\cap X_i)=0$ for $i>n$.  There may also exist $E_k$ such that $\nu(E_k\cap X_i)=0$ for $i\ge n$, but these will not contribute to the sums in the measure axiom equation:

$$\mu\left(\bigcup_kE_k\right)=\left(n,\nu\left(\bigcup_k\left( E_k\cap X_n\right)\right)\right)= \left(n,\sum_k \nu(E_k\cap X_n)\right)$$
$$=\sum_k \left(n, \nu(E_k\cap X_n)\right)=    \sum_k(\mu(E_k)). $$

If $\level(\mu\left(\bigcup_{k\in I} E_k\right ))$ is undefined, then $\mu\left(\bigcup_{k\in I} E_k\right ) =0$ or  $\mu\left(\bigcup_{k\in I} E_k\right ) =\Infty$.  In the first case, $\nu(\left(\bigcup_{k\in I} E_k\right )\cap X_i)=0$ for all $i$, which implies that $\nu(E_k\cap X_i)=0$ for all $k$ and $i$.  So $\mu(E_k)=0$ for all $k$, and the measure axiom equation holds again.   In the second case,  $\mu\left(\bigcup_{k\in I} E_k\right ) =\Infty$, which implies $\nu((\bigcup_{k\in I} E_k ) \cap X_i)>0$ for infinitely many $i$, which in turn implies that for infinitely many $i$ there exists a $k=k_i$ such that $\nu(E_{k_i}\cap X_i)>0$.    It follows that the set of levels $\level(\mu(E_{k_i})$ is unbounded, which implies  $\sum_k(\mu(E_k))=\Infty$.
\end{example}

Associated to an $\bar\SB$-measure $\mu$, we will define a collection of ordinary extended-real valued measures, namely a measure $\mu_i$ associated to each level $i$.  

\begin{defn}   Suppose $\mu$ is an $\bar\SB$- measure on $X$.   We define {\it associated $\bar\reals$ measures} $\mu_i$, $i=0,1,2,\ldots$, as follows.  If $E\subset X$ is a measurable set and $\mu(E)\ne0$, $\mu(E)\ne\Infty$,  let $\ell=\level(\mu(E))$.     Then define 

$$\mu_i(E)
=\begin{cases}\real(\mu(E))& \mbox{if } i=\ell \\
                     \infty & \mbox{if } i<\ell \mbox{ or if }\mu(E)=\Infty\in \bar\SB \\
                     0& \mbox{if } i>\ell \mbox{ or if }\mu(E)=0\in \bar\SB.
                                            \end{cases}
$$

We say the $\bar\SB$-measure $\mu$ has {\it total height $h$} if $\level(\mu(X))=h$.  We say $\mu_i$ is {\it trivial} if for every measurable $E$, either $\level(\mu(E))\ne i$, or $\mu(E)=\Infty$, or $\mu(E)=0$.
\end{defn}

For a fixed measurable $E$, the sequence $\mu_i(E)$ is a decreasing sequence, typically first a constant sequence of $\infty$'s, then a positive element, then a constant sequence of $0$'s.   In exceptional cases, we have a constant sequence of 0's, or a constant sequence of $\infty$'s.  

We must show that $\mu_i$ is a measure:

\begin{lemma}   The function $\mu_i$ on measurable sets, defined above in terms of an $\bar\SB$-measure $\mu$, is a positive $\bar\reals$-measure.
\end{lemma}

\begin{proof}  Let $E_k$, $k=1,2,\ldots$ be a sequence of disjoint measurable sets.    Let $\ell=\level(\mu(\cup_kE_k))$ if  $\mu(\cup_kE_k)\ne 0$, $\mu(\cup_kE_k)\ne \Infty$ and let $\ell_k=\level(\mu(E_k))$ if $\mu(E_k)\ne0$, $\mu(E_k)\ne\Infty$.  If $\mu(\cup_kE_k)\ne 0$, $\mu(\cup_kE_k)\ne \Infty$, since $\mu(\cup_k E_k)=\sum_k\mu(E_k)$, we have  $\level(\mu(E_k))=\ell_k\le \ell$ whenever $\mu(E_k)\ne 0,\Infty$, and in fact $\ell=\max\{\ell_k\}$.   Applying the definition of $\mu_i$ to $\cup_kE_k$ we have

$$\mu_i(\cup_kE_k)
=\begin{cases}\real(\mu(\cup_kE_k))& \mbox{if } i=\ell \\
                     \infty & \mbox{if } i<\ell  \mbox{ or if }\mu(\cup_kE_k)=\Infty \in \bar\SB\\
                     0& \mbox{if } i>\ell \mbox{ or if }\mu(\cup_kE_k)=0\in \bar\SB
                                            \end{cases}
$$
We want to show that $\mu_i(\cup_kE_k)=\sum_k\mu_i(E_k)$.  

We first consider the case $i=\ell$.  In this case $\mu_i(\cup_kE_k)=\real(\mu(\cup_kE_k))=\real(\sum_k\mu(E_k))$.  Since $\ell=\max\{\ell_k\}$, $\real(\sum_k\mu(E_k))=\sum_k\mu_i(E_k)$, and we are done.  If $i<\ell$, $i<\ell_k$ for some $k$, so $\mu_i(E_k)=\infty$ for some $k$, so $\sum_k\mu_i(E_k)=\infty=\mu_i(\cup_kE_k)$.   If $i>\ell$, then $i>\ell_k$ for all $k$, so $\sum_k\mu_i(E_k)=0=\mu_i(\cup_kE_k)$.

It remains to consider the possibilities that $\mu(\cup_kE_k)=0$ or $\mu(\cup_kE_k)=\Infty$.   In the former case all $\mu(E_k)=0$ since $\mu(\cup_k E_k)=\sum_k\mu(E_k)$.   It follows that for any $i$ and all $k$, $\mu_i(E_k)=0$ and $\mu_i(\cup_kE_k)=0$, so $0=\mu_i(\cup_kE_k)=\sum_k\mu_i(E_k)$.  In the latter case, namely the case that $\mu(\cup_kE_k)=\Infty$, either the sequence $\level(\mu(E_k))$ is unbounded or some $\mu(E_k)=\Infty$, again because $\mu(\cup_k E_k)=\sum_k\mu(E_k)$.  In either sub-case, for some $k$, $\mu_i(E_k)=\infty$.  So $\mu_i(\cup_kE_k)=\infty=\sum_k\mu_i(E_k)$.
\end{proof}

The measure $\mu$ can be recovered from the sequence $\{\mu_i\}$.  Namely, to find $\mu(E)$ given all $\mu_i(E)$, if all $\mu_i(E)=0$, then $\mu(E)=0$.  If all $\mu_i(E)=\infty$, then $\mu(E)=\Infty$.  Otherwise let $\ell$ be the greatest $i$ such that $\mu_i(E)>0$, then $\mu(E)=(\ell,\mu_\ell(E))\in \SB$.   With the correct interpretation of special cases, this amounts to saying that $\mu(E)=\sum_i(i,\mu_i(E))$.   We describe the reconstruction of the measure $\mu$ from the associated measures $\mu_i$ in greater detail below in Lemma \ref{Reconstruction}, in a more general setting.

 If one of the measures $\mu_i$ is trivial in the sense that there is no $E\in \Sigma$ such that $\level(\mu(E))=i$, and if there are non-trivial measures $\mu_j$ for some $j>i$ and for some $j<i$ then in some sense the finite height measure is equivalent to another one, in which there are no trivial $\mu_i$ between non-trivial $\mu_j$'s.   In this case we could decrease by 1 the level of every $\mu_j$ for every finite $j>i$ to decrease the number of trivial levels between non-trivial levels.   
 
 \begin{defn}\label{LevelAdjustment}  Suppose $\mu$ is an $\bar\SB$-measure on $X$.  Suppose for some $i$, $\mu_i$  is trivial in the sense that there is no $E\in \Sigma$ such that $\level(\mu(E))=i$.  Then $\mu$ is {\it equivalent via level omission} to the measure $\nu$, where $\nu$ is obtained from $\mu$ by shifting some levels:  $ \nu_j=\mu_j$ for $j<i$;  $ \nu_j=\mu_{j+1}$ for $j\ge i$.   If there are no trivial measures $\mu_i$ at levels between the levels of non-trivial $\mu_j$, we say $\mu$ is a {\it proximal} measure.
\end{defn}
 
 Informally, one could try to interpret $\bar\SB$-measures as follows.   For simplicity, suppose $\mu$ is an $\bar\SB$-measure of total height $h$ , then a set $E$  with $\level(\mu(E))=h$ might be visible to the naked eye.  To see a set with $\level(\mu(E))=h-1$ one might need a microscope.  For lower levels one would need stronger and stronger microscopes.   However, each stronger microscope would not just magnify by some finite factor, but by an infinite factor.

We now want to develop the same ideas introduced above for $\bar\SB$-measures in a much more general setting, using a semi-algebraic structure $\KB$ obtained by insertion or extended insertion.   This will make it possible to understand measures with values in $\bar \OB_r$, for example, using induction.
 
 \begin{defn} \label{AssociatedMeasure} Suppose $\KB=\integers \ens \LB$ ($\KB=\naturals_0 \ens \LB$), where $\LB$ has a greatest element $\infty$, and suppose $\mu$ is a $\KB$-measure on a measurable space $(X,\Sigma)$.  We define {\it associated $\LB$-measures} as follows.   If $E\subset X$ is a measurable set and $\mu(E)\ne0$, $\mu(E)\ne \Infty$, let $\ell =\level(\mu(E))$.     Then define   $$\mu_i(E)
=\begin{cases}\real(\mu(E))& \mbox{if } i=\ell  \\
                     \infty & \mbox{if } i<\ell \mbox{ or if }\mu(E)=\Infty\in \KB \\
                     0& \mbox{if } i>\ell \mbox{ or if }\mu(E)=0\in \KB.
                                            \end{cases}
$$
where $i\in \integers$ ( $i\in \naturals_0$.)  We say $\mu_i$ is {\it trivial} if for every measurable $E$, $\mu_i(E)=0$ or $\mu_i(E)=\infty$.   \end{defn}

 \begin{remark}   The associated $\LB$-measure $\mu_i$ defined above is trivial if and only if for any measurable set $E$, either $\level(\mu(E))\ne i$ or $\mu(E)=0$ or $\mu(E)=\Infty$.

 \end{remark}

Again, notice that for a fixed measurable $E$ satisfying $\mu(E)\ne\Infty$, $\mu(E)\ne 0$, the sequence $\mu_i(E)$ is decreasing, beginning with a constant sequence of $\infty$'s, followed by a last positive $\mu_i(E)\in \LB$ (which could be $\infty$), followed by a constant sequence of $0$'s.  We can therefore prove that each associated $\mu_i$ is actually a measure as we did in the special case of $\KB=\bar\SB$.  Notice that $\mu(E)=\Infty$ if and only if every $\mu_i(E)=\infty$;  while $\mu(E)=0$ if and only if every $\mu_i(E)=0$.

\begin{lemma}   \label{VerifyMeasure} The function $\mu_i$ on measurable sets, defined above in terms of a $\KB$-measure $\mu$, is an $\LB$-measure.
\end{lemma}

Given a sequence of $\LB$-measures with properties similar to those of a sequence of associated measures, we can construct a $\KB$-measure, $\KB=\integers\ens\LB$.   

\begin{lemma}  \label{Reconstruction} Suppose $X$ is a measurable space, and suppose $\LB$ is an ordered abelian semigroup with a greatest element $\infty$, and suppose $\KB=\integers\ens\LB$ ($\KB=\naturals_0\ens\LB$).  Suppose $\{\nu_i\}$, $i\in \integers$  ($i\in \naturals_0$), is a sequence of $\LB$-measures on $X$.   Suppose also that for any measurable set $E$ the sequence $\nu_i(E)$ satisfies the following condition:     

\begin{addmargin}[20pt]{0pt}
(\star)  If not all $\nu_i(E)=0$ and not all $\nu_i(E)=\infty$,  then there exists $\ell$ such that $\nu_\ell(E)\ne 0$, and $\nu_i(E)=0$ for $i>\ell$, $\nu_i(E)=\infty$ for $i<\ell$.  
\end{addmargin}

\noindent (a) Then there is a $\KB$-measure $\mu$ with associated measures $\nu_i$.   The measure $\mu$ is defined by $\mu(E)=(\ell,\nu_\ell(E))$ if not all $\nu_i(E)=0$ and  not all $\nu_i(E)=\infty$.  In case all $\nu_i(E)=0$, we define $\mu(E)=0$.  In case all $\nu_i(E)=\infty$, we define $\mu(E)=\Infty$.

\noindent (b) $\mu(E)=\sum_i (i,\nu_i(E))$, where $=\sum_i (i,0)$ is interpreted as $0\in \KB$ and $=\sum_i (i,\infty)$ is interpreted as $\Infty\in \KB$.
\end{lemma}

\begin{proof}  We must show that $\mu$ is a measure with associated measures $\nu_i$.  

First we show it is a measure. If $E$ is the empty set, $\nu_j(E)=0$ for all $j$, so $\mu(E)=0$.  To verify countable additivity, suppose $\{E_k\}$, $k=1,2,3\ldots$, is a sequence of disjoint measurable sets.   If not all $\nu_i(\cup_kE_k)$ are $0$ and not all $\nu_i(\cup_kE_k)$ are $\infty$, then $\mu(\cup_kE_k)=(\ell,\nu_\ell(\cup_kE_k))$ where $\ell$ is the largest $i$ with $\nu_i(\cup_kE_k)>0$.   Since $\nu_i$ is a measure, $\nu_i(\cup_kE_k)=\sum_k \nu_i(E_k)$, so $\ell$ is the greatest $i$ such that $\sum_k \nu_i(E_k)>0$.  Hence $\ell$ is the greatest $i$ such that there is a $k$ with $\nu_i(E_k)>0$.   This means that for some $s$, $\nu_\ell(E_s)>0$, but for all $k$ and all $i>\ell$, $\nu_i(E_k)=0$.  In terms of the definition of $\mu$, this means $\level(\mu(E_s))=\ell$ and for all $k$ $\level(\mu(E_k))\le \ell$.   It follows that $\mu(\cup_kE_k)=(\ell,\nu_\ell(\cup_kE_k))=(\ell,\sum_k\nu_\ell(E_k))=\sum_k(\ell,\nu_\ell(E_k))=\sum_k\mu(E_k)$.  The last equality holds since if $\level(\mu(E_k))<\ell$, the summand $\mu(E_k)$ does not contribute to the sum, and $\level(\mu(E_k))\le \ell$.

Now let $\mu_i$ be an associated measure for $\mu$.  We must show $\mu_i=\nu_i$.   For fixed measurable $E$, in the general case $\mu(E)=(\ell,\nu_\ell(E))$.  From the definition of associated measures $\mu_i(E)=\residue(\mu(E))=\nu_i(E)$ if $i=\ell$.  If $i>\ell$, $\nu_i(E)=0$ from the properties of the sequence $\nu_i(E)$, while $\mu_i(E)=0$ because $i>\level(\mu(E))$.   If $i<\ell$, $\nu_i(E)=\infty$ from the properties of the sequence $\nu_i(E)$, while $\mu_i(E)=\infty$ by definition of $\mu_i$.  In the special case where all $\nu_j(E)=0$, we defined $\mu(E)=0$, so $\mu_i(E)=0=\nu_i(E)$.   In the special case where all $\nu_j(E)=\infty$, we defined $\mu(E)=\Infty$ so $\mu_i(E)=\infty=\nu_i(E)$.
\end{proof}

Lemmas \ref{VerifyMeasure} and \ref{Reconstruction} show that the measure theories we have defined, with values in certain ordered algebraic structures, can in some sense, be completely understood in terms of classical measure theory.   However, since, for example, an $\bar\OB$-measure corresponds to a bi-infinite sequence of classical measures which are related to each other in a certain way, we see that an $\bar\OB$-measure can be more subtle. 

  Using a recursive analysis, we shall see that a measure $\mu$ with values in $\bar\OB_r$ corresponds to a sequence of sequences of sequences of...of $\bar\reals$-measures.
 
Definitions we have made before carry through to   $\KB=\integers \ens \LB$ (or $\KB=\naturals_0 \ens \LB$, $\KB=\naturals_0 \ins \LB$, $\KB=\integers \ins \LB)$, and therefore also to $\KB$ obtained by repeated (extended) insertion operations.

\begin{defn}\label{ShiftDef} Suppose $\KB=\integers \ens \LB$ (or $\KB=\naturals_0 \ens \LB$), where $\LB$ has a largest element $\infty$.  We will make definitions with notation for $\KB=\integers \ens \LB$, but they also apply to $\KB=\naturals_0 \ens \LB$.

The  {\it total height} of a $\KB$-measure $\mu$ is {\it infinite} if $\level(\mu(E))$ takes infinitely many values as $E$ varies over measurable sets, otherwise the {\it total height} is $\max\{\level(\mu(E)):E \mbox{ measurable}\}-\min\{\level(\mu(E)):E \mbox{ measurable}\}$.  The {\it set of (non-trivial) levels of $\mu$} is $\levels=\{\level(\mu(E)): E \mbox{ is measurable and }$ $\mu(E)\ne0,\ \mu(E)\ne\Infty\}$.    A {\it trivial level for $\mu$} is any level not in $\levels$.  A {\it gap level for $\mu$} is a level $g$ not in $\levels$ such that there exist $r,s\in \levels$ with $r<g<s$.

Suppose $g$ is a gap level.  Then $\mu$ is {\it equivalent via level omission} to the measure $\nu$, where $\nu$ is obtained from $\mu$ by omitting a level:   if $\mu(E)=0$ or $\mu(E)=\Infty$, we define $\nu(E)=\mu(E)$.   Otherwise, let $\ell=\level(\mu(E))$, and define

$$\nu(E)
=\begin{cases}\mu(E)& \mbox{if } \ell<g\\
			(\ell-1,\real(\mu(E)) & \mbox{ if }\ell> g.\\
			
                                                                \end{cases}$$
In general $\mu$ and $\nu$ are {\it equivalent by level omission} if one is obtained from the other by a (possibly infinite) sequence of level omissions.  If $\mu$ has no gap levels, we say $\mu$ is a {\it proximal} measure, so $\levels$ is a set of consecutive integers.

The measure $\mu$ is equivalent via {\it level shift (by $k$ levels)} to $\nu$ if $\nu$ is defined in terms of $\mu$ as follows:   If $\mu(E)=0$ or $\mu(E)=\Infty$, we define $\nu(E)=\mu(E)$.   Otherwise,  let $\ell=\level(\mu(E))$, and define $\nu(E)=(\ell+k,\real(\mu(E)) $.   If a measure has finitely many non-trivial levels, it is {\it finite height}.   If it is equivalent via level shift and level omission to a measure with non-trivial levels $\levels=\{0,1,2\ldots\}$, it is called {\it infinite height}; if it is equivalent to a measure with non-trivial levels $\levels=\{\ldots,-2,-1,0\}$, then it is {\it infinite depth}; if it is equivalent to a measure with non-trivial levels $\integers$, then it is called {\it bi-infinite}.

The definitions also apply to  $\mu$ a measure with values in $\KB=\integers \ins \LB$ (or $\KB=\naturals_0 \ins \LB$) since we can simply regard such a $\mu$ as a measure with values in the larger  $\KB=\integers \ens \LB$ (or $\KB=\naturals_0 \ens \LB$).
\end{defn}

We can describe the effect of the level omission on associated measures.  If $g$ is a gap  level then the measure $\nu$ obtained from $\mu$ by level omission has associated measures  $ \nu_i=\mu_i$ for $i<g$;  $ \nu_i=\mu_{i+1}$ for $i\ge g$.  The effect of a level shift (by $k$) of $\mu$ yielding $\nu$  gives associated measures $\nu_i(E)=\mu_{i-k}(E)$.  

Note that the number of levels in $\levels$, when finite, is one more than the total height.  The definition is designed to be consistent with definitions for finite depth foliations, laminations, etc.   Using level shifts and level omissions, we can replace any finite height $\KB$-measure ($\KB=\integers \ins \LB$ or $\KB=\naturals_0 \ins \LB$) by a proximal measure with  $\mu_i$ trivial for $i<0$ and the $\mu_0$ non-trivial.   Or viewing the measure as being finite depth, we could arrange that $\mu_i$ is trivial for $i>0$, $\mu_0$ is non-trivial, with finitely many negative levels $i$ where $\mu_i$ is non-trivial.   If there are infinitely many levels $i$, we can similarly use level omissions and shifts to arrange that $\levels=\naturals_0$, or $\levels=\integers\setminus \naturals$, or $\levels=\integers$.

\hop\hop

\noindent{\it Borel measures}
\hop

Given a Hausdorff topological space $X$, a Borel measure assigns a measure to each set in a $\sigma$-algebra generated by open sets.   The values of a positive Borel measure lie in $[0,\infty]\subset \bar\reals$. 

\begin{defn}\label{BorelMeasureDef}   Suppose $\LB$ is any ordered abelian semigroup with the lub property.   A {\it Borel $\LB$-measure} $\nu$ on a Hausdorff topological space $X$ assigns an element $\nu(E)$ of $\LB$ to each Borel set $E$ such that  $\nu$ satisfies the measure axioms (Definition \ref{LMeasureDef}).
\end{defn}

Again, typically we use an $\LB$ with the lub property, but this is not necessary.

There is a simple way of constructing $\bar\OB$-measures and  $\bar\OB_r$-measures which is especially useful for Borel measures.  

\begin{lemma}  Suppose $X$ is a measurable space.   Suppose $X_i$, $i\in \integers$ ($i\in \naturals_0$) is a sequence of disjoint measurable subsets and suppose $\rho_i$ is an $\LB$-measure on $X_i$.   Then\\ $\mu(E)=\sum_i(i,\rho_i(E\cap X_i))\in \KB=\integers\ens\LB$ defines a $\KB$-measure on $X$ if $(i,0)$ is interpreted as $0\in\KB$.
\end{lemma}

\begin{proof}  We verify the measure axioms.   The fact that $\mu(\emptyset)=0$ follows immediately.   If $E_k$, $k=1,2,\ldots$ is a sequence of disjoint measurable sets, 
$$\mu(\cup_kE_k)=\sum_i\left(i,\rho_i(X_i\cap[\cup_kE_k])\right)=\sum_i\left(i,\sum_k\rho_i(E_k\cap X_i)\right)=$$

$$=\sum_i\sum_k(i,\rho_i(E_k\cap X_i))=\sum_k(\sum_i(i,\rho_i(E_k\cap X_i))=\sum_k\mu(E_k).$$
\end{proof}

\begin{defn}  Suppose $\mu$ is a $\KB$-measure on a measurable space $X$, where $\KB=\integers\ens\LB$  ( $\KB=\naturals_0 \ens\LB$).  Then $\mu$ is {\it decomposable} if there exist disjoint measurable sets $X_i$ and $\LB$-measures $\rho_i$ on $X_i$ such that $\mu(E)=\sum_i(i,\rho_i(E\cap X_i))$ where  $(i,0)$ is interpreted as $0\in\KB$.
\end{defn}

\begin{remark}  It should be clear that  $\rho_i=\mu_i|_{X_i}$, where $\mu_i$ is the $i$-th associated measure for $\mu$.  Also, this definition excludes the possibility that there exist measurable $E$ with $\mu(E)=\Infty$ and $\rho_i(E)=0$ for all $i$.   
\end{remark}

We should give an example of a non-decomposable $\bar\OB$-measure.

\begin{example} Regarding $\bar\OB$ as $\integers\ens[0,\infty]$ we define a $\bar\OB$-measure on $\reals$ by $\mu(E)=(0,\nu(E))$ for uncountable Lebesgue measurable sets $E$, where $\nu(E)$ is Lebesgue measure.   If $E$ is countable $\mu(E)=(-1,|E|)$ where $|E|$ is the number of elements in $E$, or $\infty$ in case $E$ is countably infinite.  
\end{example}

The motivation for the definition of ``decomposable," as well as other properties of measures we will describe below, all come from certain measures which arise as transverse measures for codimension-1 laminations.  Transverse measures are defined in terms of measures on transversals, which of course are homeomorphic to intervals in $\reals$.  The following is a motivating example.

\begin{example} \label{CantorExample} Let $X=[0,1]$ and let $X_{i}$ be subsets, $i=-d,-d+1\ldots,-2,-1,0$ as follows.   $X_0$ is the usual middle thirds Cantor set in the unit interval.   $X_{-1}$ is a countable union of middle thirds Cantor sets, with one such Cantor set inserted in the interior of each middle third removed to construct $X_0$.  Inductively, assuming we have constructed $X_{i}$, let $X_{i-1}$ be a countable union of Cantor sets, with one such Cantor set inserted in the interior of every middle third removed to construct $X_{i}$.   Clearly $Y_i=\cup_{j\ge i}X_j$ is a closed set.   Each Cantor set in each $X_{i}$ will have a measure $\rho_{i}$ coming from the construction of the Cantor set in the unit interval.  Thus for large $|i|$ the measure on $X_{i}$ will be large compared to the Lebesgue measure in the unit interval.  We define an $\bar\OB$-measure $\mu$ on $X$ by $\mu(E)=\sum_i(i,\rho_i(E\cap X_i))$.  
Then $\rho_i$ is a locally finite (Radon),  full support measure on $X_i$, and $\mu$ has a kind of local finiteness property, namely it locally has values in $\PB$.  Furthermore, from our definition of $\mu$ it is immediate that $\mu$ is decomposable as a measure with values in $\bar\OB=\integers\ens[0,\infty]$.

Finally, we verify that the measure $\mu$ has an inner regularity property.   Namely, for any measurable $E$, $\mu(E)=\sup\{\mu(K):K\subset E \mbox{ is compact}\}$.   To see this, first observe that $\rho_i$ is locally finite, hence Radon and inner regular.  So $\rho_i(E\cap X_i)=\sup\{\rho_i(K):K\subset E\cap X_i \mbox{ compact}\}$.  On the other hand, in the definition  $\mu(E)=\sum_i(i,\rho_i(E\cap X_i))$, if $\mu(E)\ne 0$, there must be a maximum $i$ such that $\rho_i(E\cap X_i)$ is non-zero, which we call $\ell$.   Then $\mu(E)=(\ell,\rho_\ell(E\cap X_\ell))$ and we know $\rho_\ell(E\cap X_\ell)=\sup\{\rho_\ell(K):K\subset E\cap X_\ell \mbox{ compact}\}$.   Hence also $\mu(E)=(\ell,\sup\{\rho_\ell(K):K\subset E\cap X_\ell \mbox{ compact}\})=\sup\{\mu(K):K\subset E\cap X_\ell \mbox{ compact}\}\le\sup\{\mu(K):K\subset E\mbox{ is compact}\}$.  On the other hand, it is obvious that $\sup\{\mu(K):K\subset E\mbox{ is compact}\}\le\mu(E)$.   It remains to consider the case that $\mu(E)=0$, but in that case it is obvious that  $\mu(E)=\sup\{\mu(K):K\subset E \mbox{ is compact}\}$.  \end{example}

\begin{defn}  A {\it Borel decomposable $\bar\OB$-measure} $\mu$ is a decomposable Borel measure on a topological space $X$ with the following properties:
\begin{tightenum}
\item $X_i$ is a sequence (finite, infinite, or bi-infinite) of disjoint subspaces of $X$.
\item $\rho_i$ is a $\bar\reals$-measure of full support on $X_i$. 
\item $\mu$ is defined by $\mu(E)=\sum_i(i,\rho_i(E\cap X_i))$.

\item $Y_i=\cup_{j\ge i} X_j$ is a closed subspace.
\end{tightenum}

We will say a Borel decomposable $\bar\OB$-measure $\mu$ on a locally compact Hausdorff space is a {\it $(\PB,\bar\OB)$-measure} if $\rho_i$ is locally finite and $X_i$ is $\sigma$-compact.

The measure is {\it inner regular} if for any measurable $E$, $$\mu(E)=\sup\{\mu(K): K\subset E\cap X_i \mbox{ for } i\in \integers \mbox{ and } K \mbox{ compact}\}$$
$$=\sup\{\mu(K): K\subset E \mbox{ is compact}\}.$$
\end{defn}

The following lemma says that a $(\PB,\bar\OB)$-measure is determined by its values in $\PB$.

\begin{lemma} \label{POLemma} Suppose $\mu$ is a $(\PB,\bar\OB)$-measure.  In other words, suppose $\mu$ is a Borel decomposable $\bar\OB$-measure on a locally compact, Hausdorff topological space $X$, which decomposes into locally finite positive $\bar\reals$-measures $\rho_i$ on $\sigma$-compact subspaces $X_i$.  Then:

\noindent (a) $\mu$ is inner regular, and is determined by its values in $\PB$.

\noindent (b) The subsets $X_i$ and the measures $\rho_i$ are uniquely determined by $\mu$.
\end{lemma}

\begin{proof}  We begin by showing that $X_i$ is locally compact and Hausdorff.  This follows from standard point set topology as follows.   Since $X$ is locally compact and Hausdorff, the closed subspace $Y_{i}$ is locally compact and Hausdorff.   Open subspaces of locally compact and Hausdorff spaces are also locally compact.   Therefore, since $X_i$ is open in $Y_{i}$, it is locally compact and Hausdorff.   Since $X_i$ is $\sigma$-compact, we conclude by the Riesz Representation Theorem that the locally finite $\bar\reals$-measure $\rho_i$ is inner regular, a  Radon measure.  

Next we show that $\mu(K)\in \PB$ for $K\subset X_i$ compact.  If $K\subset X_i$ is compact, then because $\rho_i$ is locally finite, $\rho_i(K)$ is finite.  Then $\mu(K)=\sum_j(j,\rho_j(K\cap X_j))=(i,\rho_i(K))\in \PB$.

To prove (a), consider an arbitrary measurable $E$. We know $\mu(E)=\sum_i(i,\rho_i(E\cap X_i))$.   The subspace $X_i$ is locally compact and Hausdorff, so since $\rho_i$ is   inner regular, $\rho_i(E\cap X_i)=\sup\{\rho_i(K):K\subset E\cap X_i \mbox{ compact}\}$.  Suppose first that $\mu(E)\ne 0$ and that there is a maximum $i$ such that $\rho_i(E\cap X_i)$ is non-zero, say $i=\ell$.   Then $\mu(E)=(\ell,\rho_\ell(E\cap X_\ell))$ and we know $\rho_\ell(E\cap X_\ell)=\sup\{\rho_\ell(K):K\subset E\cap X_\ell \mbox{ compact}\}$.   Hence also $\mu(E)=(\ell,\sup\{\rho_\ell(K):K\subset E\cap X_\ell \mbox{ compact}\})=\sup\{\mu(K):K\subset E\cap X_\ell \mbox{ compact}\}\le\sup\{\mu(K):K\subset E\mbox{ is compact}\}$.  On the other hand, it is obvious that $\sup\{\mu(K):K\subset E\mbox{ is compact}\}\le\mu(E)$.   
It remains to consider two cases: the case that $\mu(E)=0$ and the case that the set of levels $i$ such that $\rho_i(E\cap X_i)\ne 0$ is not bounded above.  If $\mu(E)=0$, it is obvious that  $\mu(E)=\sup\{\mu(K):K\subset E \mbox{ is compact}\}$.  If there is an infinite increasing sequence $\{i_j\}$ such that $\rho_{i_j}(E\cap X_{i_j})\ne 0$, then $\mu(E)=\sum_i(i,\rho_i(E\cap X_i))=\Infty$, and there exist compact $K_{i_j}\subset E\cap X_{i_j}$ with $\rho_i(K_{i_j})>0$, which implies $\sup\{\mu(K):K\subset E \mbox{ is compact}\}=\Infty$.  This completes the proof of inner regularity.  The measure is then determined by its values on compact sets, hence by its values in $\PB$.

 For statement (b) we claim that $Y_i$ is determined by $\mu$ as the set of points $x\in X$ such that every neighborhood $V$ of $x$ satisfies $\level(\mu(V))\ge i$ or $\mu(V)=\Infty.$  Thus $Y_i$ is a kind of support of the ``part of $\mu$ at level $\ge i$," and is closed.  To justify the claim, observe that if $x\in Y_i$ then $x\in X_j$ for some $j\ge i$.   Therefore for every neighborhood $U$ of $x$ in $X_j$,  $\rho_j(U)>0$ because $\rho_j$ has full support in $X_j$.   If $V$ is a neighborhood of $x$ in $X$, then $U=V\cap X_j$ satisfies $\rho_j(U)>0$  which implies $\level(\mu(V))\ge i$.  Conversely, if $x$ is a point with the property that every neighborhood $V$ of $x$ satisfies $\level(\mu(V))\ge i$, then $x\notin X\setminus Y_i$ because $X\setminus Y_i$ is an open neighborhood of $x$.  Since the $Y_i$'s are determined by $\mu$, so are the $X_i$'s.   Finally $\rho_i$ is determined for $E\subset X_i$ as $\rho_i(E)=\residue(\mu(E))$.  
 \end{proof}

An analogue of Lemma \ref{POLemma} for $\bar\OB_r$ would be desirable, so we develop the necessary ideas.

\begin{defn}  Suppose $\KB=\integers\ens \LB$, where $\LB$ is an ordered abelian semigroup and suppose $X$ is a locally compact Hausdorff space.   A {\it Borel decomposable $\KB$-measure} $\mu$ on $X$ is a Borel measure which decomposes into $\LB$-measures $\rho_i$ as follows:
\begin{tightenum}
\item $X_i$ is a sequence (finite, infinite, or bi-infinite) of disjoint subspaces of $X$.
\item $\rho_i$ is an $\LB$-measure of full support on $X_i$.  
\item $\mu$ is defined by $\mu(E)=\sum_i(i,\rho_i(E\cap X_i))$.
\item $Y_i=\cup_{j\ge i} X_j$ is a closed subspace.
\end{tightenum}

A Borel $\bar\OB_r$-measure $\mu$ on a space $X$ is {\it recursively Borel decomposable} if $\mu$ is Borel decomposable into $\bar\OB_{r-1}$-measures, each of the  $\bar\OB_{r-1}$-measures is Borel decomposable into  $\bar\OB_{r-2}$-measures, and so on, until we end with locally finite $\bar\reals$-measures.    The end result is a collection of Radon full support positive $\bar\reals$-measures $\rho_{i_1i_2\cdots i_r}$ on disjoint subspaces $X_{i_1i_2\cdots  i_r}$.

We will say a Borel decomposable $\bar\OB_r$-measure on a locally compact Hausdorff space is a {\it $(\PB_r,\bar\OB_r)$-measure} if it decomposes into locally finite $\bar\reals$-measures $\rho_{i_1i_2\cdots i_r}$ on $\sigma$-compact subspaces $X_{i_1i_2\cdots  i_r}$.

The measure is {\it inner regular} if for any measurable $E$, $$\mu(E)=\sup\{\mu(K): K\subset E\cap X_{i_1i_2\cdots  i_r} \mbox{ for }{(i_1,i_2,\cdots  i_r)\in \integers^r} \mbox{ and } K \mbox{ compact}\}$$
$$=\sup\{\mu(K): K\subset E \mbox{ is compact}\}.$$
\end{defn}

In the following the reader should recall the notation for elements of $\bar\OB_r$ given in Definition \ref{ONotation}, as well as the definition of the level functin $\level$ on $\bar\OB_r$.

\begin{lemma} \label{PORLemma} Suppose $X$ is a locally compact, Hausdorff space and suppose $\mu$ is a  $(\PB_r,\bar\OB_r)$-measure.  In other words, suppose $\mu$ is a recursively Borel decomposable $\bar\OB_r$-measure on a locally compact, Hausdorff topological space $X$, yielding locally finite $\bar\reals$-measures $\rho_{i_1i_2\cdots i_r}$ on disjoint $\sigma$-compact subspaces $X_{i_1i_2\cdots  i_r}$.  Then:

\noindent (a) $\mu$ is inner regular, and is determined by its values in $\PB_r$.

\noindent (b)  The subsets  $X_{i_1i_2\cdots  i_r}$ and the measures $\rho_{i_1i_2\cdots i_r}$ are uniquely determined by $\mu$.

\end{lemma}

\begin{proof} The proof is an inductive version of the proof of Lemma \ref{POLemma}.

 We begin by showing that $X_{i_1i_2\cdots  i_r}$ is locally compact and Hausdorff.  The first Borel decomposition of $\mu$ yields disjoint sets $X_i$ with $\OB_{r-1}$-measures  $\rho_i$.   Since $Y_{i}$ is a closed subspace, it is locally compact, and since $X_i$ is open in $Y_{i}$ it is also locally compact.   Now for each $i$ the measure $\rho_i$ is Borel decomposable, yielding a Borel decomposition $X_i=\cup_jX_{ij}$ with an $\OB_{r-1}$ measure $\rho_{ij}$ for each $(i,j)$.   The $X_{ij}$'s are locally compact for the same reason as before; they can be obtained by passing to closed and open subspaces.  Continuing the proof inductively, we conclude that each  $X_{i_1i_2\cdots  i_r}$ is locally compact and by the definition of a recursively Borel decomposable $\bar\OB_r$, the $\bar\reals$-measure $\rho_{i_1i_2\cdots i_r}$ is locally finite and therefore Radon and inner regular.  

Now we show $\mu(K)\in \PB_r$ if $K\subset X_{i_1i_2\cdots  i_r}$ is compact. If $K\subset  X_{i_1i_2\cdots  i_r}$ is compact, then because $\rho_{i_1i_2\cdots i_r}$ is locally finite, $\rho_{i_1i_2\cdots  i_r}(K)$ is finite.  Then 
$$\mu(K)=\sum_{(j_1,j_2\ldots, j_r)}((j_1,j_2\ldots, j_r),\rho_{j_1j_2\ldots j_r}(K\cap X_{j_1j_2\ldots j_r}))=((i_1,i_2\ldots, i_r),\rho_{i_1i_2\cdots  i_r}(K))\in \PB_r.$$

To prove (a), consider an arbitrary measurable $E$. We know 

\begin{equation}\label{eqn:mudef}
\mu(E)=\sum_{(j_1,j_2\ldots, j_r)}\left((j_1,j_2\ldots, j_r),\rho_{j_1j_2\ldots j_r}(E\cap X_{j_1j_2\ldots j_r})\right)
\end{equation}

\noindent  Since $\rho_{j_1j_2\ldots j_r}$ is   inner regular, 
$$\rho_{j_1j_2\ldots j_r}(E\cap X_{j_1j_2\ldots j_r})=\sup\{\rho_{j_1j_2\ldots j_r}(K):K\subset E\cap X_{j_1j_2\ldots j_r} \mbox{ compact}\},$$  
\noindent where the supremum could be infinite.

Suppose first that $\mu(E)\ne 0$ and that there is a maximum $(j_1,j_2\ldots, j_r)$ (with the lexicographical ordering) such that $\rho_{j_1j_2\ldots j_r}(E\cap X_{j_1j_2\ldots j_r})$ is non-zero, say $(j_1,j_2\ldots, j_r)=(i_1,i_2\ldots, i_r)$.   Then $\mu(E)=((i_1,i_2\ldots, i_r),\rho_{i_1i_2\cdots  i_r}(E\cap X_{i_1i_2\cdots  i_r}))$ and we know 
$$\rho_{i_1i_2\cdots  i_r}(E\cap X_{i_1i_2\cdots  i_r})=\sup\{\rho_{i_1i_2\cdots  i_r}(K):K\subset E\cap X_{i_1i_2\cdots  i_r} \mbox{ compact}\}.$$

\noindent    Hence also
 $$\mu(E)=((i_1,i_2\ldots, i_r),\sup\{\rho_{i_1i_2\cdots  i_r}(K):K\subset E\cap  X_{i_1i_2\cdots  i_r} \mbox{ compact}\})$$
$$=\sup\{\mu(K):K\subset E\cap X_{i_1i_2\cdots  i_r} \mbox{ compact}\}\le\sup\{\mu(K):K\subset E\mbox{ is compact}\}.$$
 \noindent On the other hand, it is obvious that $\sup\{\mu(K):K\subset E\mbox{ is compact}\}\le\mu(E)$.   
It remains to consider two cases: the case that $\mu(E)=0$ and the case that there is no maximum level $(j_1,j_2\ldots, j_r)$ such that $\rho_{j_1j_2\ldots,j_r}(E\cap X_{j_1j_2\ldots,j_r})\ne 0$.  If $\mu(E)=0$, it is obvious that  $\mu(E)=\sup\{\mu(K):K\subset E \mbox{ is compact}\}$.  In the other remaining case, there are infinitely many non-zero terms in Equation \ref{eqn:mudef}, and the sum must equal one of the infinities in $\bar\OB_r$.    In that case there also exist compact sets  $K_{j_1j_2\ldots j_r}\subset E\cap  X_{j_1j_2\ldots j_r}$ at the same levels, with $\rho_{j_1j_2\ldots j_r}(K_{j_1j_2\ldots j_r})>0$.  This implies $\sup\{\mu(K):K\subset E \mbox{ is compact}\}=\mu(E)$, the same infinity.  This completes the proof of inner regularity.  The measure is then determined by its values on compact sets in subspaces $X_{i_1i_2\cdots  i_r}$, hence by its values in $\PB_r$.

We prove statement (b) by induction, using essentially the same proof as in Lemma \ref{POLemma}.  For $r=1$, $\bar\OB_r=\bar\OB$, we proved the statement in the previous lemma.   Suppose we have proved the statement for $\bar\OB_{r-1}$.   Now suppose $\mu$ is a recursively Borel decomposable $\bar\OB_r$- measure on $X$.   So $\mu(E)=\sum_i(i,\rho_i(E\cap X_i))$, where $\rho_i$ is a $\bar\OB_{r-1}$ measures on $X_i$ for each $i$.   Again, $Y_i=\cup_{j\ge i} X_j$ and we claim that $Y_i$ is determined by $\mu$ as the set of points $x\in X$ such that every neighborhood $V$ of $x$ satisfies $\level(\mu(V))\ge i$.  To justify the claim, observe that if $x\in Y_i$ then $x\in X_j$ for some $j\ge i$.   Therefore for every neighborhood $U$ of $x$ in $X_j$,  $\rho_j(U)>0$ because $\rho_j$ has full support in $X_j$.   If $V$ is a neighborhood of $x$ in $X$, then $U=V\cap X_j$ satisfies $\rho_j(U)>0$  which implies $\level(\mu(V))\ge i$.  Conversely, if $x$ is a point with the property that every neighborhood $V$ of $x$ satisfies $\level(\mu(V))\ge i$, then $x\notin X\setminus Y_i$ since $X\setminus Y_i$ is a neighborhood of $x$, and clearly $\level(\mu(X\setminus Y_i))<i$.  Since the $Y_i$'s are determined by $\mu$, so are the $X_i$'s.   Finally $\rho_i$ is determined for $E\subset X_i$ by the induction hypothesis and $\rho_i(E)=\residue(\mu(E))$.  
\end{proof}

\section{Laminations with transverse $\LB$-measures, \\ $\LB$ metric spaces, and $\LB$ trees.}

Suppose $L$ is an essential codimension-1 lamination in a compact surface $S$ with $\chi(S)<0$.   Up to minor modifications, essential laminations in a surface with $\chi(S)<0$ can be realized as geodesic laminations for any chosen hyperbolic structure on $S$.  We will usually assume that we have chosen a hyperbolic structure.
The trees dual to lifts of arbitrary essential laminations  are more general trees called ``order trees."  There is a definition in
 \cite{DGUO:EssentialLaminations}, but we give a different definition here.  I am not sure who first defined these; I first heard about them from Peter Shalen.  In this section, we consider laminations with transverse measures in ordered algebraic structures, and investigate the corresponding additional structure on dual order trees.   Many of the ideas extend to codimension-1 laminations in higher dimensional manifolds.
 
 \begin{defn}  An {\it order tree} is a set $\T$ together with a subset $[x,y]$, called a {\it segment}, associated to each pair of elements $x,y$.   Each segment $[x,y]$ has a linear order such that $x$ is the least element in $[x,y]$ and $y$ is the greatest element.  We allow {\it trivial segments} $[x,x]$.  The set of segments should satisfy the following axioms:
 
 \begin{tightenum}
 \item The segment $[y,x]$ is the segment $[x,y]$ with the opposite order.
\item The intersection of segments $[x,y]$ and $[x,z]$ is a segment $[x,w]$, with $[x,w]\subset [x,y]$ and $[x,w]\subset [x,z]$.  
 
\item  If two segments intersect at a single point, $[x,y]\cap[y,z]=\{y\}$ then the union is a segment  $[x,z]$.
 \end{tightenum}

 \end{defn}

Clearly $\reals$-trees and $\bar\reals$-trees are also order trees.  Using (ii) amd (iii) one can show that the intersection of two segments is a segment.

Suppose now that $\LB$ is any ordered abelian semigroup.  

\begin{defn} An {\it $\LB$-metric} on a set $X$ is a function $d:X\times X\to \LB$ satisfying the usual axioms for a metric.  An {\it $\LB$-metric space} is the set $X$ together with an $\LB$-metric.

An {\it $\LB$-tree} is an order tree with an $\LB$ metric.
\end{defn}

We will assume that $\LB$ is an ordered abelian semigroup and consider Borel $\LB$-measures on segments of trees.

\begin{defn}  Suppose $\T$ is an order tree and suppose $\nu$ is a Borel $\LB$-measure $\nu$ on the disjoint union of segments of $\T$ with the property that  if $[x,y]$ and $[z,w]$ are segments, and $[x,y]\cap [z,w]=[u,v]$, then for any measurable set $E\subset [u,v]$, $\nu(E)$ is the same no matter which segment ($[x,y]$,$ [z,w]$, or $[u,v]$) we use to evaluate the measure.  (The measure agrees on intersections of segments.)   We say $\nu$ is an {\it $\LB$-measure on $\T$}; $\T$ with the measure  $\nu$ is called an {\it $\LB$-measured tree}.

 The $\LB$ measure on an order tree is {\it non-atomic}  if the measure of a single point in a segment is always $0$.  It has {\it full support} if it has full support on the disjoint union of segments.
 
If the measure $\mu$ is a $(\PB,\bar\OB)$-measure (a $(\PB_r,\bar\OB_r)$-measure) on the disjoint union of segments, we say $(\T,\mu)$ is a $(\PB,\bar\OB)$-measured (a $(\PB_r,\bar\OB_r)$-measured) tree.
\end{defn}

\begin{lemma}  Suppose $\T$ is an $\LB$-measured tree with a non-atomic full support measure $\nu$.  Then $\T$ is an $\LB$-metric space with metric $d(x,y)=\nu([x,y])\in \LB$.
\end{lemma}

\begin{proof}  Because $\nu$ has no atomic measures on points, we conclude $d(x,y)=\nu([x,y])=0$ if and only if $x=y$.  To verify the triangle inequality, observe that if $x,y,z$ are points in the tree, by axiom (iii) for order trees, $[x,y]\cap [x,z]=[x,w]$ for some $w$, so $[y,w]\cup [w,z]=[y,z]$ by axiom (iii).  Hence $d(y,z)=\nu ([y,z])=\nu([y,w])+\nu([w,z])\le \nu([y,x])+\nu([x,z])=d(y,x)+d(x,z)$, because $[y,w]\subset [y,x]$ and $[w,z]\subset [x,z]$.
\end{proof}

\begin{defn}\label{OrderDef} Suppose $L$ is a codimension-1  lamination in a manifold $M$.    We say a {\it homotopy of transversals} is a homotopy from a compact transversal $T_0$ to another transversal $T_1$ through transversals $T_t$, $0\le t\le 1$ such that each endpoint of $T_t$ remains in the same leaf or in the same complementary component of $L$ as $t$ varies.  The homotopy of transversals gives an identification of $T_0\cap L$ with $T_1\cap L$, hence a measure $\mu_0$ on $T_0$ with support $T_0\cap L$ can be identified with a measure $\mu_1$ with support $T_1\cap L$, and we say $\mu_1$ is the {\it invariant image} of  $\mu_0$ under the homotopy of transversals.

A {\it transverse $\LB$-measure for $L$} is an assignment of a value $\mu(R)\in \LB$ to every Borel subset $R$ of a closed transversal $T$ of the lamination which has support $L\cap T$.  In particular, $\mu$ assigns a measure $\mu(T)$ to every compact transversal $T$.  The assignment must be {\it invariant}:  If $T_t$ is a homotopy of transversals, then the measure $\mu$ on $T_t$ is the invariant image of $\mu$ on $T_0$.
 
 In case $L$ is an essential lamination in a closed surface $S$ satisfying $\chi(S)<0$, the {\it order tree dual to the lift $\tilde L$ of $L$ to the universal cover $\tilde S$ of $S$} is the set of closures of complementary regions of $\tilde L$ union non-boundary leaves.  A {\it segment} $[x,y]$ is the set elements of $\T$ intersected by a closed oriented efficient transversal $T$ for $\tilde L$ with order coming from the order on the transversal.  \end{defn}

\begin{proof}[Proof of Proposition \ref{DualMeasureProp}] Suppose $x,y$ are points in $\T$, representing leaves or complementary components $X$ and $Y$.  A geodesic $\gamma$ from a point in $X$ to a point in $Y$ gives an efficient transversal, hence a segment $[x,y]$ in $\T$.  The uniqueness of this segment is also easy to verify:  Suppose $\beta$ is another efficient transversal from $X$ to $Y$.  Choose a geodesic segment $\omega$ in $Y$ joining the final point of $\gamma$ in $ Y$ to the final point of $\beta$ in $ Y$, and similarly choose a geodesic segment $\rho$ joining the initial in point $\gamma\cap X$ to the initial point of $\beta$ in $X$.  Since $\gamma\omega\bar\beta\bar\rho$ is null homotopic, we obtain a map $h:R\to \tilde S$ of a square $R$ to $\tilde S$ whose sides are mapped to $\gamma,\beta, \rho,\omega$.  The null-homotopy $h$ can be homotoped such that it is transverse to $\tilde L$, and can then be further homotoped such that the induced lamination on $R$ consists of leaves joining opposite sides of $R$ mapped to $\gamma$ and $\beta$.  

    We verify the order tree axioms:  (i) is true by construction, $[y,x]$ is $[x,y]$ with the opposite order, coming from a transversal with the opposite orientation.  
    
    For (ii), consider oriented geodesic transversal segments $\gamma$ from a point in $X$ to a point in $Y$, and $\beta$ from a point in $X$ to a point in $Z$.  We may choose $\gamma$ and $\beta$ so that they do not intersect.   Choose a geodesic segment $\omega$ joining the final point of $\gamma$ to the final point of $\beta$, and choose a geodesic segment $\rho$ joining the initial point of $\gamma$ in $ X$ to the initial point of $\beta$ in $ X$.   The simple closed $\gamma\omega\bar\beta\bar\rho$ bounds a rectangular disk $R$ in $\tilde S$ and $\tilde L\cap R$ is a lamination in $R$ which is transverse to two opposite sides $\gamma$ and $\beta$ with $\rho\subset\bdry R$.  Consider the set of leaves of $\tilde L\cap R$ joining $\gamma$ to $\beta$.  This includes at least $\rho$ and it must be closed.  So there is a largest element $w$ in $[x,y]$ which is also in $[x,z]$.
    
For property (iii), suppose $[x,y]$ and $[y,z]$ are (non-trivial) segments in $\T$ with $[x,y]\cap [y,z]=\{y\}$.    Representing $[x,y]$ by an oriented geodesic segment $\beta$ and $[yz]$ by an oriented geodesic segment $\gamma$ whose initial point is the final point of $\beta$, we see that $\beta\cup \gamma$ must be an embedded path.   It follows that $\beta\cup \gamma$ can be regarded as a transversal, representing $[x,z]$.  

Now that we know that $\T$ is an order tree, it is easy to show it is an $\LB$-tree.  The transverse $\LB$-measure $\mu$ for  $L$ yields a transverse measure $\tilde \mu$ for $\tilde L$, which is a measure on transversals.  Since transversals are identified with segments of $\T$, we have a measure $\nu$ on the segments.  Invariance of the measure $\tilde \mu$ gives an $\LB$-measure $\nu$ on the disjoint union of segments of $\T$.  If there are no leaves of $L$ with atomic measure, there are no points with atomic measure in (the segments of) $\T$, which shows that $d(x,y)=\nu([x,y])$ defines an $\LB$-metric on $\T$, using also the fact that $\nu$ has full support.  

The action of $\pi_1(S)$ on the universal cover $\tilde S$ yields an action on the $\LB$-tree.
\end{proof}

There is a version of Proposition \ref{DualMeasureProp}, not included here,  which applies to any essential lamination in a 3-manifold with a transverse $\LB$-measure.  Whether we are working with codimension-1 laminations in surfaces or 3-manifolds (or even higher dimensional manifolds), it is useful to use branched manifolds with invariant $\LB$-weight vectors to describe the lamination.   We now define branched manifolds and a number of concepts related to branched manifolds.

\begin{defns} \label{BranchedManifoldDef}  A {\it branched $m$-manifold} is a space with the property that a neighborhood of each point can be described as a quotient $B_\ell$ obtained from a stack of $k$ unit $m$-balls of the form $D_i=\{(x,i)\in\reals^{m+1}:x\in \reals^m, |x|\le 1\}=D\times \{i\}$, $i=1,\ldots,k$, where $D$ is the unit ball.   For each $D_i$, $i=1,\ldots,k-1$, we choose a ``half-ball" $H_i$ in the unit ball $D$ in $\reals^m$ cut from $D_i$ by a smooth properly embedded $m-1$-ball $\gamma_i$ passing through the origin.   Then we identify $H_i\times \{i\} \subset D_i$ with $H_i\times \{i+1\} \subset D_{i+1}$.  The resulting quotient $B_\ell$, with an appropriate smooth structure in which $D_i$'s are smooth, models the branched $m$-manifold locally and is itself a branched $m$-manifold with boundary.  The {\it branch locus} of any branched manifold is the set of points which do not have neighborhoods locally modeled on a ball, i.e. the set of non-manifold points.  Assuming the branch locus of a branched manifold $B$ is well-behaved, such that $B$ has a cell-complex structure in which the branch locus is a subcomplex of dimension $m-1$, the {\it sectors} of $B$   are the the completions in $B$ of the components of the complement of the branch locus.

 Referring to the construction of a local model of $B_\ell$ above, the {\it fibered neighborhood} $V(B_\ell)$ of $B_\ell$ is constructed locally from a stack of $k$ products of the form $D\times [i,i+0.5]=\{(x,h)\in\reals^{m+1}:x\in \reals^m, |x|\le 1, h\in [i,i+0.5] $, $i=1,\ldots,k$.   For $i=1,\ldots,k-1$ we identify $H_i\times \{i+0.5\}\subset D\times [i,i+0.5]$ with $H_i\times \{i+1\}\subset D\times [i+1, i+1+0.5]$.  The products are foliated by intervals, which can be combined to give a foliation by {\it fibers} of the local model $V(B_\ell)$ of the fibered neighborhood.  The {\it cusp manifolds} in the local model are the images under the quotient map of $\gamma_i\times\{i+0.5\}\subset D\times [i,i+0.5]$.  See Figure \ref{OrderNbhd}.

\begin{figure}[H]
\centering
\scalebox{0.8}{\includegraphics{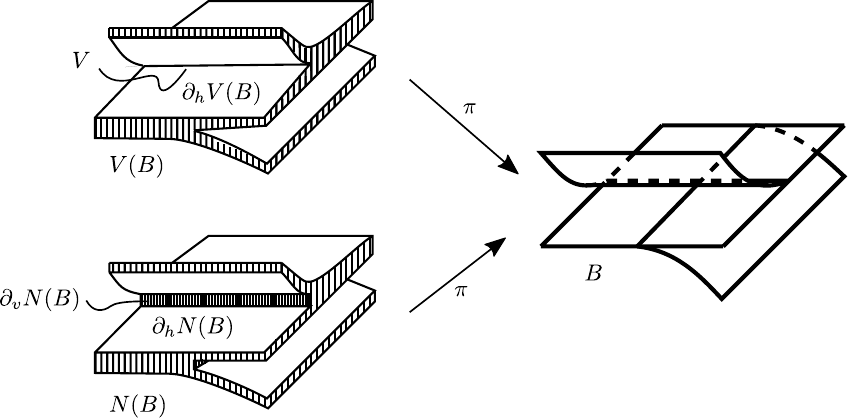}}
\caption{\small Fibered neighborhoods of branched manifolds.}
\label{OrderNbhd}
\end{figure}

 For an arbitrary codimension-1 branched manifold $B\embed M$, where $M$ has dimension $n$ and $B$ has dimension $m=n-1$,  $V(B)$ is a regular neighborhood with a quotient projection map $\pi:V(B)\to B$ such that every point in $B$ has a neighborhood  $B_\ell$ such that $\pi\inverse (B_\ell)=V(B_\ell)$ as described above, with $\pi\inverse(x)$ a fiber of $V(B_\ell)$ for each $x\in B_\ell\subset B$.   A lamination $L$ is {\it carried} by a branched manifold $B$ if it can be embedded in a {\it fibered neighborhood $V(B)$} of $B$ transverse to fibers of $V(B)$.   The {\it branch locus} of $B$ and {\it cusp manifolds} $V$ for $V(B)$ are defined in terms of the local models.   For closed manifolds $M$ and closed branched manifolds, the {\it horizontal boundary} $\bdry_hV(B)$ is the completion of $\bdry V(B)\setminus V$.  For manifolds with boundary  $\bdry_hV(B)$ is the completion of  $\bdry V(B)\setminus (V\cup \bdry M)$.    
    
 There is another kind of fibered neighborhood $N(B)$ of a branched manifold $B$ very similar to $V(B)$, in which the cusp manifolds of $V(B)$ are replaced by $I$-bundles as shown in the figure, which are denoted $\bdry_vN(B)$.  The horizontal boundary $\bdry_hN(B)$ is more clearly identifiable in $N(B)$.   
 
 If $N(B)$ and $N(B')$ are fibered neighborhoods of two different branched manifolds $B$ and $B'$, then $B'$ is a {\it splitting} of $B$ if $N(B)=N(B')\cup J$, where $J$ is an $I$-bundle over an $m$-manifold $F$ (usually with boundary), all of $\bdry_h J$ is attached to $\bdry_h N(B')$, and some of the remainder of $\bdry J$ is attached to $N(B')$ as follows:   If $p:J\to F$ is the projection, $\bdry_vJ$ denotes $q\inverse (\bdry F)$, $\delta$ is a submanifold of dimension $m-1=n-2$ of $\bdry F$, and $\beta$ is the complementary submanifold of $\bdry F$, then $q\inverse(\delta)$ (part of $\bdry_v J$) is attached to $\bdry_v N(B')$ (identifying fiber to fiber) and $q\inverse(\beta)$ is contained in $\bdry_vN(B)$.  Each fiber of $N(B)$ intersects $J$ in fibers of $J$.   If $B'$ is a splitting of $B$ there is a {\it projection map}
$p:B'\to B$ defined in terms of fibers.   If $x'\in B'$, and the fiber $(\pi')\inverse(x')\subset \pi\inverse(x)$, then $p(x')=x$.
 If $B'$ is a splitting of $B$, and $w'$ is an invariant weight vector on $B'$ with weights in some ordered abelian semigroup $\LB$, then $w'$ determines an invariant weight vector $w$ for $B$, where the weight on a sector of $B$ containing $x$ is the sum of the weights at points of $p\inverse(x)$.   The invariant weight vector $w$ is the weight vector {\it induced by $w'$}.
 
 We have not given a formal definition of an {\it invariant weight vector} on a branched manifold $B$.   If $B$ is a ball, there is just one sector, and the invariant weight vector is any element of $\LB$.  The local model $B_\ell$ has a splitting consisting of $k$ disjoint balls $D_i$, $i=1,2,\ldots k$.  An invariant weight vector on this disjoint union of balls assigns any weight to each of the $D_i$.  Then an invariant weight vector on $B_\ell$ is a weight vector induced by the weight vector on the disjoint union of balls $D_i$.
 
If a lamination $L$ is carried by $B$, we may assume $L$ is transverse to fibers of $N(B)$ and $\bdry_hN(B)$ is contained in $L$.  If  only one leaf of $L$ passes through some fiber of $N(B)$, one cannot isotope $L$ to ensure $\bdry_hN(B)$ is contained in $L$, but after replacing a leaf by the boundary of a suitably tapered regular neighborhood of itself, one can achieve this.   The {\it interstitial $I$-bundle for $L$ in $N(B)$} is the completion of $N(B)\setminus L$, with fibers mapping into fibers of $N(B)$.   A {\it splitting of $B$ along $L$} is a splitting $B'$ such that $N(B)=N(B')\cup J$ with $J$ mapped to the interstitial bundle respecting fibers.
 
Suppose now that the branch locus of $B$ is generic.   In terms of a local model $V(B_\ell)$, this means that the projections to $D$ of the cusp manifold gives a family of smooth $m-1$-balls in general position in the $m$-ball $D$.   If $B$ carries a lamination $L$ with a transverse $\LB$-measure $\mu$, one obtains a {\it weight vector induced by $L$} assigning a value of $\LB$ to each sector of $B$:  The weight on a sector $Z$ of $B$ is $\mu(T)$ where $T=\pi\inverse(\{z\})$ is a fiber over a point $z$ in the interior of the sector.  The weight vector induced by $\LB$ satisfies certain equations called {\it branch equations}:    Regarding the branch locus as a cell-complex, for every $m-1$-cell $\beta$ in the branch locus, there are three sectors $Z_0$, $Z_1$, and $Z_2$ adjacent to $\beta$ such that $Z_0$ and $Z_1$ are joined smoothly at $\beta$, and similarly $Z_0$ and $Z_2$ are joined smoothly along $\beta$.   If $w_0$, $w_1$, and $w_2$ are the weights on $Z_0$, $Z_1$, and $Z_2$ respectively, we obtain a branch equation $w_1+w_2=w_0$.    All the branch equations arise in this way.    A weight vector with entries in $\LB$ is called an {\it invariant weight vector} if it satisfies the branch equations.
  \end{defns}

Recall that the points in the
 branch locus of a train track (branched 1-manifold) are called {\it switch points}.

\begin{proof}[Proof of Theorem \ref{FiniteDepthThm}]  Suppose $L$ is a multi-level codimension-1 measured lamination as defined in the introduction, Definition \ref{FiniteDepthLam}, in terms of the sequence $L_i$, $i\in \integers$, of measured laminations with measure $\mu_i$ for $L_i$.   In the proof, we assume that some of the $L_i$'s may be empty.   By reindexing to omit empty laminations and shifting indices, we may suppose that we have a finite sequence of nonempty laminations, a bi-infinite sequence of nonempty laminations, a sequence with $L_i\ne\emptyset$ if and only if $i\le 0$ (infinite depth), or a sequence with $L_i\ne\emptyset$ if and only if $i\ge 0$ (infinite height).  Recall $L_i\subset M_i$,  where $M_i=M\setminus\bigcup_{j>i}L_j$. 

We must show that $L$ has a $(\PB,\bar\OB)$ transverse measure.  Given a transversal $X$ for the lamination, we let $X_i=X\cap L_i$ for $i\in \integers$.   The full support transverse measure $\mu_i$ on $L_i$ gives a full support Radon measure $\rho_i$ on $X_i$, and we see that we can define a Borel decomposable measure  $\mu|_X$ as we did in the previous section.  Abusing notation by writing $\mu$ instead of $\mu|_X$, for any measurable $E\subset X$ we define an $\bar \OB$-measure $\mu(E)=\sum_i(i,\rho_i(E\cap X_i))$.   Then cleary $\mu$ is a $(\PB,\bar\OB)$-measure on $X$, with all the properties described in Lemma \ref{POLemma}. 

  It is not difficult to check that $\mu$ is a transverse measure.  Invariance of the transverse measure $\mu$ follows from the invariance of the transverse measures $\mu_i$.

For the converse, suppose $L\embed M$ has a transverse $(\PB,\bar\OB)$-measure $\mu$.  This means that if $X$ is a transversal for $L$, there is a measure $\mu|X$ (for simplicity, just $\mu$) which is a $(\PB,\bar \OB)$-measure.   In particular, the fact that $\mu$ is Borel decomposable means that $\mu$ is defined by $\mu(E)=\sum_i(i, \rho_i(E\cap X_i))$, where the subspaces $X_i$ of $X$ have the property that $Y_i=\cup_{j\ge i} X_j$ is closed and $\rho_i$ is a Radon measure on $X_i$.   We let $L_i$ be the union of leaves that intersect $X_i$ for some transversal $X$, with a transverse measure $\mu_i$ equal to $\rho_i$ on a given transversal $X$.  The lamination $L_i$ and the transverse measure $\mu_i$ is well-determined by this condition, not depending on the choice of transversal.  This is because the invariant image of $\rho_i$ by a homotopy of transversals must be the same as the $\rho_i$ in the image transversal by the part of Lemma \ref{POLemma} which guarantees the uniqueness of the decomposition into $X_i$ and $\rho_i$.  Because $Y_{i+1}$ is closed for every transversal $X$, $\cup_{j>i}L_j$ is a lamination.   Invariance of the transverse measure $\mu$ guarantees invariance of transverse measures $\mu_i$. 
 
Now suppose $L$ is essential in a closed surface $M$, or a surface with cusps of finite type with $\chi(M)<0$.  If $L$ had infinitely many measured levels, using an Euler characteristic argument we would conclude that for sufficiently large $j$, $M_j$ is a product. But we are assuming no leaf of $L_j$ isotopes into a leaf in $\bdry M_j$, and the fact that $L$ is essential implies that there are no 2-dimensional Reeb components, so $L$ cannot have complementary products.  \end{proof}

To understand laminations with transverse $(\PB_r,\bar\OB_r)$-measures via Theorem \ref{RecursiveMultiThm}, we first give a precise definition of $\integers^r$ multi-level measured laminations.  Intuitively, a lamination $L$ of this type is one which can be constructed recursively as follows.   The lamination $L$ can be expressed as $\cup_{i\in \integers} L_i$ where $L_i$ is a $\integers^{r-1}$ multilevel measured lamination, with $L_i$ a lamination in the complement of $\cup_{j>i}L_j$ such that $\cup_{j\ge i}L_j$ is also a lamination.  Some of the $L_i$ may be empty.  The recursion ends with $\integers^0$ multi-level measured laminations, which are just measured laminations, or one could end with  $\integers^1$ multi-level measured laminations, which are multi-level measured laminations as defined before.  

\begin{defn}  Let $\integers^r$ denote the product with the lexicographic ordering.   A {\it $\integers^r$ multi-level measured lamination $L$ in $M$} is one which can be written as $\displaystyle L=\cup_{i\in \integers^r}L_i$, where $\cup_{j\le i}L_j$ is a lamination in $\displaystyle M\setminus \cup_{j>i}L_j$, and where $\displaystyle \cup_{j>i}L_j$ is  required to be a lamination.  We also require that no leaf of $L_i$ be isotopic to a leaf of  $\displaystyle \cup_{j>i}L_j$.  (All subscripts in $\integers^r$.)
\end{defn}

Now we can prove the theorem in the introduction relating $(\PB_r,\bar\OB_r)$-measures to $\integers^r$ multi-level measured laminations.  Before we prove the theorem, here is an example.

 \begin{example}  We describe a lamination in the plane with countably many leaves, carried by a train track $B$ shown in Figure \ref{ZRExample}.   The first leaf $L_{(0,0)}$ is a line in the plane.   It has atomic measure $1$ at level $(0,0)$.   Other leaves $L_{(-1,j)}$ have atomic measure $1$ at levels $(-1,j)$.  We show a few leaves in the figure.   The induced $\PB_2$-weight vector on $B$ is shown.  In order to make the different leaves non-isotopic, we introduce some topology in each digon, as indicated by $X$'s, for example we could introduce a hole.
 \end{example}

\begin{figure}[H]
\centering
\scalebox{0.6}{\includegraphics{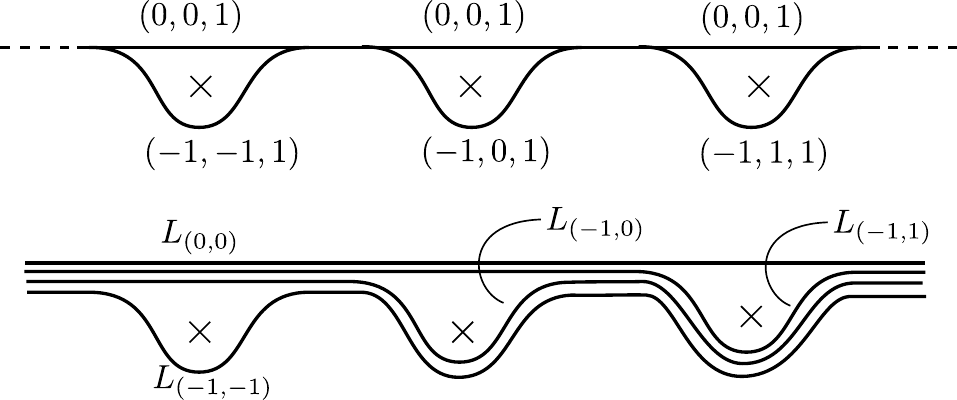}}
\caption{\small Example of a $\integers^2$ multi-level lamination.}
\label{ZRExample}
\end{figure}

\begin{proof}[Proof of Theorem \ref{RecursiveMultiThm}]  The proof is by induction.  We have proved the result for multi-level measured laminations: namely the structure of a multi-level measured lamination for $L\embed M$ corresponds to a transverse $(\PB,\bar\OB)$ measure.   We observe that a $\integers^1$ multi-level measured lamination is the same as a multi-level measured lamination, so our previous result starts the induction.    

Now suppose $\displaystyle L=\bigcup_{i=-\infty}^\infty L_i$ where $L_i$ a  $\integers^{r-1}$ multi-level measured lamination, with each $L_i$ embedded in $\displaystyle M_i= M\setminus \bigcup_{j>i}L_j$. 
We must show that $L$ has a $(\PB_r,\bar\OB_r)$- transverse measure, assuming $L_i$ is a $(\PB_{r-1},\bar\OB_{r-1})$ measured lamination..  Given a  transversal $X$ for the lamination, we let $X_i=X\cap L_i$ for $i\in \integers$.   The full support transverse $(\PB_{r-1},\bar\OB_{r-1})$-measure $\mu_i$ on $L_i$ gives a full support Borel  $(\PB_{r-1},\bar\OB_{r-1})$- measure $\rho_i$ on $X_i$, and we see that we can define a Borel decomposable measure  $\mu|_X$ as we did in the previous section.  Abusing notation by writing $\mu$ instead of $\mu|_X$, for any measurable $E\subset X$ we define an $\bar \OB_r$-measure $\mu(E)=\sum_i(i,\rho_i(E))$.   Then cleary $\mu$ is a $(\PB_r,\bar\OB_r)$-measure on $X$, with all the properties described in Lemma \ref{PORLemma}. 

  It is not difficult to check that $\mu$ is a transverse measure.  Invariance of the transverse measure $\mu$ follows from the invariance of the transverse measures $\mu_i$.

For the converse, suppose $L\embed M$ has a transverse $(\PB_r,\bar\OB_r)$-measure $\mu$.  This means that if $X$ is a transversal for $L$, there is a measure $\mu|_X$ (for simplicity, just $\mu$) which is a $(\PB_r,\bar \OB_r)$-measure.   In particular, the fact that $\mu$ is recursively Borel decomposable means that $\mu$ is defined by $\mu(E)=\sum_i(i, \rho_i(E\cap X_i))$, where the subspaces $X_i$ of $X$ partition $X\cap L$ and have the property that $Y_i=\cup_{j\ge i} X_j$ is closed and $\rho_i$ is a recursively Borel decomposable $(\PB_{r-1},\bar\OB_{r-1})$-measure on $X_i$.   We let $L_i$ be the union of leaves that intersect $X_i$ for some transversal $X$, with a transverse measure $\mu_i$ equal to $\rho_i$ on a given transversal.  The lamination $L_i$ and the transverse measure $\mu_i$ is well-determined by this condition, not depending on the choice of transversal.  This is because the invariant image of $\rho_i$ by a homotopy of transversals must be the same as the $\rho_i$ in the image transversal by the part of Lemma \ref{PORLemma} which guarantees the uniqueness of the decomposition into $X_i$ and $\rho_i$.  Because $Y_{i+1}$ is closed for every transversal $X$, $\cup_{j>i}L_j$ is a lamination.   Invariance of the transverse measure $\mu$ guarantees invariance of transverse measures $\mu_i$.  
\end{proof}

If $B$ is compact with finitely many sectors, finitely many branch equations suffice to determine whether a weight vector is invariant.  When $\KB=[0,\infty)\subset \reals$, an invariant weight vector uniquely determines a measured lamination carried by $B$.  This is not  true  for arbitrary $\KB$, but we prove Proposition \ref{TrackExistenceProp}, stated in the introduction, which solves the realization problem for train tracks in surfaces.

 \begin{proof}[Proof of Proposition \ref{TrackExistenceProp}]  If one can inductively define a sequence of splittings $\ldots B_{3}\prec B_{2}\prec B_{1}\prec B_0=B$, each $B_i$ with an weight vector $\bw_i$, such that the $\PB$ weight vector on $B_i$ induces the weight vector on $B_{i-1}$, then the inverse limit of $(B_i,\bw_i)$ defines a measured lamination provided the inverse limit of the $B_i$'s is a lamination.  Here $B_{1}\prec B_0$ means $B_1$ is a splitting of $B_0$.   A good scheme for finding a sequence of splittings whose inverse limit actually defines a lamination is to split in a neighborhood of one segment (sector) of the train track at a time.   We fix a cell structure for $B$ dividing it into 1-cells such that each 1-cell has distinct ends.  This can be achieved by subdividing the segments of the train track.   For simplicity, we refer to the 1-cells as ``segments."   If $\sigma_2$ and $\sigma_3$ are segments 
joining smoothly to $\sigma_1$ at a switch $s$ we can split $\sigma_1$ to obtain two copies.   At the other end of $\sigma_1$ we split a little beyond the end of the segment, see Figure \ref{SplitTrack}.    It is easy to choose $\PB$-weights on the split train track $B_{i+1}$ which induce the weight on $B_i$.  There is only one situation where there is a choice, see Figure \ref{SplitTrack} (c), (d), where new weights at lower levels may be introduced.  We perform a sequence of splittings, cycling through switch points so that we ``make progress" at each switch.   If we do not introduce unnecessary new segments as in Figure \ref {SplitTrack} (d), then the number of switches does not increase.  Further, we observe that avoiding unnecessary new segments gives a canonical way of splitting, and also gives a canonical realization of the weights on $B$ by a $(\PB,\bar\OB)$-measured lamination.
 \end{proof}

\begin{figure}[H]
\centering
\scalebox{0.8}{\includegraphics{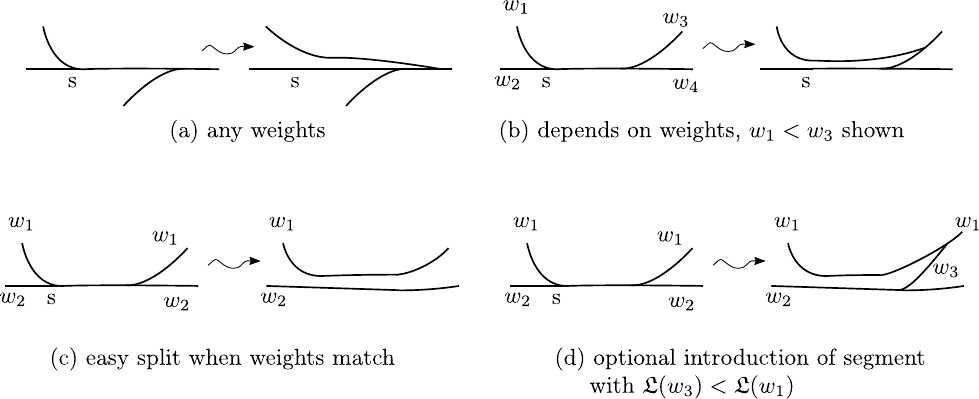}}
\caption{\small How to split train track with $\PB$ weights to obtain $\OB$-measured lamination.}
\label{SplitTrack}
\end{figure}

In the introduction we described the general {\it realization problem}, the problem of determining whether an invariant vector of weights in some ordered algebraic structure, assigning weights to sectors of a codimension-1 branched manifold $B$, is induced by by a lamination carried by $B$ with a transverse measure in the same ordered algebraic structure.

We now know that invariant $\PB$-weights on a branched manifold $B$ do not determine an $\bar\OB$-measured lamination uniquely.  We also know that invariant $\PB$-weights on a train track $B$ in a surface can be realized by a canonical $(\PB,\bar\OB)$-measured lamination carried by $B$.   So the next question is whether the realization problem can be solved in higher dimensions.  In particular, we can ask the following.   If $\bw$ is an  invariant $\PB$ weight vector on $B$ a branched surface embedded in a 3-manifold $M$, is there a $(\PB,\bar\OB)$-measured lamination $(L,\mu)$ (a finite depth measured lamination) which induces the weight vector $\bw$?  Without additional assumptions, the answer is no, as the following example shows.

\begin{example}  Let $B$ be a disk of contact branched surface (embedded in the 3-ball) with one additional triangular sector, forming what is called a ``twisted disk of contact," see Figure \ref{TDC}.   The $\PB$ weights indicated in the figure are not realizable.   The weights on the boundary indicate the boundary lamination must consist of two closed curves together with a leaf which spirals to limit on both of these closed curves.   But the level -1 lamination on the boundary is a spiral leaf which certainly does not bound a measured lamination in the product (disk) $\times I$, by Reeb stability. 
\end{example}

\begin{figure}[H]
\centering
\scalebox{0.6}{\includegraphics{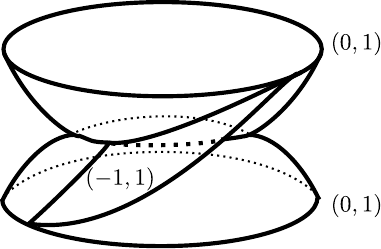}}
\caption{\small Twisted disk of contact with $\PB$ invariant weights.}
\label{TDC}
\end{figure}

It is possible that the only obstruction is related to weights in neighborhoods of surfaces of contact, or immersed surfaces of contact in $B$.   It is also possible that the only obstruction in higher dimensions is analogous.

To state a theorem addressing the realization problem in higher dimensions, we make some definitions.

\begin{defn}  Suppose $B\embed M$ is a codimension-1 branched manifold and suppose $w$ is an invariant $\PB$ weight vector on $B$.   ($B$ may not be compact and the weight vector may have infinitely many entries.)  For each $m\in \integers$, let $B_m$ be the union of sectors $Z_i$ of $B$ with weight vectors $w_i$ satisfying $\level(w_i)\ge m$.   We say $\{B_m\}$ is the {\it sequence of branched manifolds associated to $B$ with the weight vector $w$}.   \end{defn}

Some observations:  Clearly $ B_m\setminus B_{m+1}$ is a branched manifold, improperly embedded in $M$, to which $w$ assigns weights at level $m$, and these weights then determine a measured lamination carried by  $ B_m\setminus B_{m+1}$.   The problem is to show that the measured laminations carried by $ B_m\setminus B_{m+1}$ can be extended to form a measured laminations $L_m$ whose union is a multilevel measured lamination inducing the weight vector $w$ on $B$.  The branched manifold $C_m=B_m\setminus B_{m+1}$ is attached to $B_{m+1}$ along a codimension-2 branched manifold in $\bdry C_m$ to obtain $B_m$.  (If $M$ and $B$ have no boundary, $C_m$ is attached along all of its boundary branched manifold.)  
There is another way of describing the attachment.   Replace $B_{m+1}$ by its regular neighborhood $V(B_{m+1})$.  Then attach $C_m$ to $\bdry V(B)\setminus V$ on (part of) its boundary, then apply the projection which collapses fibers of $V(B_{m+1})$ to points of $B_{m+1}$.   (Recall $V$ is the cusp manifold in $\bdry V(B)$.)

\begin{defn}  We say $C_m$ is {\it smoothly attached} to $B_{m+1}$ if in the above construction $\bdry C_m$ is disjoint from the cusp curves of $V(B_{m+1})$.  Alternatively, $B_m$ is obtained from $N(B_{m+1})$ by attaching $\bdry C_m$ to $\intr(\bdry_hN(B_{m+1}))$, then applying the projection $\pi$ which replaces $N(B_{m+1})$ by $B_{m+1}$.
\end{defn}

We will give a sufficient condition for the realizability of invariant $\PB$-weight vectors on certain branched surfaces in 3-manifolds.   To do this we need the following lemma.

\begin{lemma}\label{RealizeLemma}  Suppose $Q$ is an $n$-manifold and $F$ is a manifold in $\bdry Q$.  Suppose $B$ is a branched codimension-1 manifold embedded in $Q$ (properly or not) consisting of $F$ with a branched manifold $C$ attached to $F$ (on one side of $F$, since $F\subset \bdry Q$).    Suppose we are given an invariant weight vector $w$ on $B$ with entries having values in $[0,\infty]\subset \bar\reals$, where the weights are $\infty$ on all sectors in $F$ and finite on sectors of $C$.   These weights are realizable by a canonical measured lamination (not including $F$ as leaves) in $\intr(Q)$, where the transverse measure is locally finite in the interior of $Q$, and where the measure is infinite on any open transversal in the interior whose closure intersects $F$.
\end{lemma}

\begin{proof}  Using charts, it is easy to show that the invariant weight vector on $C$ uniquely determines a measured lamination carried by $C$.   The difficulty lies in constructing a lamination on a regular neighborhood in $B$ of $F$, that is to say the intersection of a regular neighborhood in $Q$ of $F$ intersected with a regular neighborhood of $B$.   If $F$ is not connected, we replace $F$ by one of its components and assume it is connected.   Our argument applies to each component of the original $F$.   We also assume that $B$ is the regular neighborhood of a connected $F$ and $B$ is homotopy equivalent to $F$, and we may assume $Q$ is a product $F\times I$.  We assume $C$ deformation retracts to $F$.   $C$ is attached to $F$ on a set which we can regard as a transversely oriented branched manifold $\tau$ of dimension $n-2$, where the transverse orientation comes from the sense of branching where $C$ is attached.   (One can think of the case $n=3$, where $\tau$ is a train track in the surface $F$, and $\tau=\bdry C$.)  The invariant weight vector on $B$ induces an invariant weight vector on $B$ and on $C$, hence also on $\tau$, where the weights on $\tau$ are the same as the weights on the adjacent sectors in $C$, not $F$, and are therefore finite.    Let $v$ be the induced weight vector on $\tau$.  Then the transversely oriented lamination $\tau(v)$ can be interpreted as a cohomology class in $H^1(F;\reals)$.     If the class $[\tau(v)]$ is trivial, then there is a good way of realizing the weights as a Radon transverse measure on a lamination as follows.   We regard $V(B)$ as a neighborhood with the measure of a fiber over an interior point of a sector given by the weight on the sector.   For fibers of $\pi\inverse(F)$ we also identify the fiber with a subset of $\reals$.   Choose a sector $Z_0$ in $F$ at random, and identify fibers over the interior points of $Z_0$ with $[-\infty, 0]$ as shown in Figure \ref {OrderAltitude}.  We would like to model $\pi\inverse(F)$ as a subset of $F\times [-\infty,\infty)$ such that the measure on fibers of $\pi\inverse(F)$ coming from the identification extends consistently to $\pi\inverse(B)$, where fibers over the interior of a sector $Z_j$ of $C$ are identified with an interval of length $w_j$ in $\reals$.   Here $F\times \{-\infty\}=F\subset \bdry Q$.  Fibers over interior points of a sector $Z_i$ in $F$ are identified with $[-\infty,z_i]$, and we call $z_i$ an {\it altitude}.  The altitudes $z_i$ and weights $w_i$ must then satisfy {\it branch equations}.   In the figure, $z_2+w_1=z_0=0$ and $z_2+w_3=z_4$.  Thus the branch equations are just like the branch equations for weights.  When $[\tau(v)]$ is trivial as a cohomology class, we claim this kind of structure for $V(B)$ exists.

\begin{figure}[H]
\centering
\scalebox{0.6}{\includegraphics{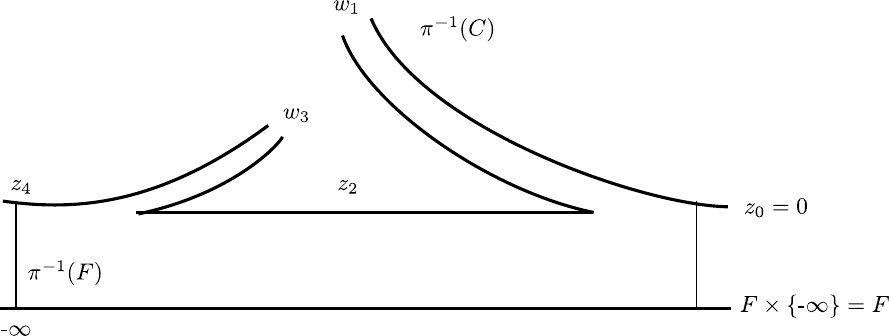}}
\caption{\small Realizing weights on $B$ as $V(B)$ with measure on fibers.}
\label{OrderAltitude}
\end{figure}

We think of the cohomology class $[\tau(v)]$ as an element of $\phi\in \hom(\pi_1(F),\reals)$.   Thus, for an element $[\gamma]\in \pi_1(F)$, making $\gamma$ transverse to $\bdry C$, $\phi(\gamma)$ is the ``net change in altitude" of the horizontal boundary in $\pi\inverse(F)$ as we traverse $\gamma$.  If we assume the cohomology class is trivial,  $\phi(\gamma)=0$ is trivial for any closed $\gamma$, so the net change in the altitude of sectors going around the loop is $0$.   In the Figure, if $\delta$ is a path in $F$ from the left side to the right, the net change of altitude $-w_3+w_1$.   Thus we have shown that if the class $[\tau(v)]$ is trivial, $B$ carries a measured lamination with infinite transverse measure on fibers of $\pi\inverse(F)$.  

If the cohomology class $[\tau(v)]$ is non-trivial we pass to the universal cover $\tilde Q$ of $Q$, containing the universal cover $\tilde B$ of $B$, which contains the universal cover $\tilde F$ of $F$.  (Recall we replaced $Q$ by the product $F\times I$.)   Letting $\tilde C$ be the lift of $C$, which attaches to $\tilde F$ on $\tilde \tau$, and letting $\tilde w$ be the lifted weight vector, which induces the weight vector $\tilde v$ on $\bdry \tilde C=\tau$, we can use the above argument to construct a measured lamination carried by $\tilde B$ induced by the same weights:   In $\tilde F$, all cohomology classes are trivial, in particular $[\tau(\tilde v)]$ is trivial, hence the weight vector $\tilde w$ can be realized by a measured lamination as before.   Covering translations preserve this measured lamination, but the measured model of $V(\tilde B)$ (analogous to the one in Figure \ref{OrderAltitude}) is not invariant, as the altitudes are not necessarily preserved by covering translations.  Applying a covering translation corresponding to $\gamma\in \pi_1(F)$ preserves the weights $w_i$, but changes the altitudes:  a fiber with altitude $z_j$ over a sector of $F$ is mapped to a fiber with altitude $z_j+ \phi(\gamma)$.

Recall that we have dealt with one component of the original $F$, but we can show in the same way that neighborhoods in $B$ of the other components of $F$ carry laminations realizing the weights.  Then the measured lamination determined by the weights on $C$ can be combined with the lamination we have constructed near each component of $F$ to obtain the required measured lamination.
\end{proof}

The simplest example of the construction in the above proof involving a non-trivial cohomology class is a spiral train track $B$ in an annulus $Q$, with $F=\bdry Q$, with weights as shown in Figure \ref{OrderSpiral}(a).  The construction yields a locally finite spiral $\reals$-measured lamination in $\intr(Q)$ with infinite transverse measure near $F=\bdry Q$.    If we include $F$ in the lamination, we can re-interpret the enlarged lamination as a $(\PB,\bar\OB)$-measured lamination, with transverse measure at level -1 limiting on both circles of $\bdry Q$ which have atomic transverse measures at level 0.  The induced weights on $B$ are shown in Figure \ref{OrderSpiral}(b).   
We  use the idea in this example to prove the following theorem.

\begin{figure}[H]
\centering
\scalebox{0.6}{\includegraphics{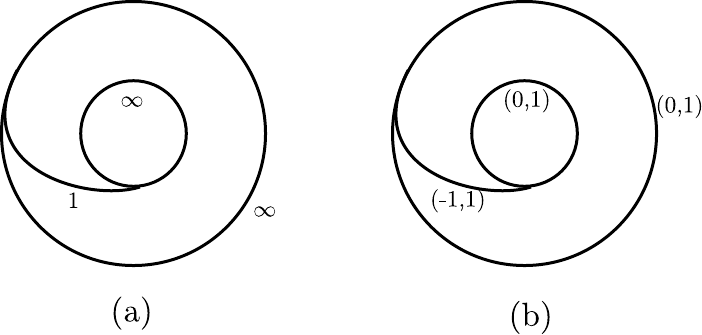}}
\caption{\small Example of $B$ with $\PB$ weights where one must pass to a cover to construct $(\PB,\bar\OB)$-measured lamination.} 

\label{OrderSpiral}
\end{figure}

\begin{thm}   \label{PORealizeThm} Suppose $B$ is a branched manifold (compact or non-compact) embedded in a 3-manifold $M$, and suppose $w$ is an invariant $\PB$ weight vector on $B$, with associated branched surfaces $B_m$.   If for every $m$, $C_m=B_m\setminus B_{m+1}$ is smoothly attached to  $B_{m+1}$, then the weight vector $w$ is realizable by a $(\PB,\bar\OB)$-measured lamination carried by $B$ and the realization is canonical.
\end{thm}

\begin{proof}   Let us begin with a special case:  Suppose $B_{q+1}=\emptyset$ for some integer $q$.   This means that the levels of weights (from the weight vector $w$) on sectors of $B$ are bounded above by $q$, which is the case when $B$ is compact.   Without loss of generality, we may suppose that $B_q$ is non-empty.   We let $w_q$ be a weight vector assigning real weights to sectors of $B_q$ derived from $w$.  Thus $w_q$ skips entries of $w$ for sectors which are not in $B_q$.   Further, if $w_j$ is a $\PB$ weight on a sector of $B_q$, we let $(w_q)_j=\residue(w_j)$, which is real.   The invariant weight vector $w_q$ of $w$ on $B_q$ then uniquely determines a $\reals$-measured lamination carried by $B_q$.  Another way of saying this is that there is an infinite sequence of splittings $B_{q,i} $, (with $i=0,1 \ldots$, and  $B_{q,0}=B_q$)  and weight vectors $w_{q,i}$ such that  $w_{q,i+1}$ on $B_{q,i+1}$ induces  the weight vector $w_{q,i}$ on  $B_{q,i} $.  The measured lamination $L_q$ carried by $B_q$ is the inverse limit of $(B_{q,i},w_{q,i}) $.  

We will prove by induction that the weight vector on $B_p$, $p\le q$ determines a $(\PB,\bar\OB)$-measured lamination carried by $B_p$, with weights in levels $p\le j\le q$.  So we assume we have a $(\PB,\bar\OB)$- measured lamination carried by $B_{p+1}$ realized as an inverse limit of branched manifolds $B_{p+1,i}$ with weight vectors $w_{p+1,i}$, again with $i=0,1,2\ldots$.   Now $C_p$ is attached to $\bdry_hN(B_{p+1})$, so that applying the projection to $N(B_{p+1})$ replaces $C_p\cup N(B_{p+1})$ by $B_p$.   In fact, because $\bdry_hN(B_{p+1})\subset \bdry_hN(B_{p+1,i})$ for each $i$ we also have $C_p$ attached to $\bdry_hN(B_{p+1,i})$, so that applying the projection to $N(B_{p+1,i})$ replaces $C_p\cup N(B_{p+1,i})$ by  a branched manifold we will call $B_{p,i}$.   We have $B_{p,i}$ constructed from $B_{p+1,i}$ with $C_p$ attached.   For sufficiently large $i$, we may assume $\bdry C_p$ is attached to $\bdry_hN(B_{p+1,i})$ far from $\bdry_ v N(B_{p+1,i})$, where `` far" could be measured combinatorially using a cell decomposition on the original $B$.   We now apply Lemma \ref{RealizeLemma} with $Q$ the closure of $M\setminus  N(B_{p+1,i})$ and with $F=\bdry_h N(B_{p+1,i})$.  The real parts of the weights on $C_p$ together with $\infty$ on components of $F$ give an invariant weight vector for $C_p\cup F$.   The lemma says that this weight vector on $C_p\cup F$ can be realized as a measured lamination which does not include $F$.   This is part of the measured lamination we want to construct at level $p$.  Again measuring distance in terms of a cell decomposition on $B$, we can split along cusp manifolds at $\bdry C_p$ to split $C_p$ away from $B_{p+1,i}$ some distance, using the measured lamination carried by $C_p\cup F$ as a guide.   (Split by removing some interstitial bundle associated to the measured lamination.)     We introduce a finite number of splittings between $B_{p,i}$ and $B_{p,i+1}$ which enlarge $C_p$ in this way by splitting away from $B_{p+1,i}$.    We do the corresponding splittings on $B_{p,i'}$ for $i'>i$, so that we still have an inverse sequence $B_{p,i}$.  We also reindex the splittings $B_{p,i}$ to include the new splitting operations introduced between $i$ and $i+1$ and between $i'$ and $i'+1$, $i'>i$.   Note that all the newly introduced splittings between $i$ and $i+1$ ($i'$ and $i'+1$) fix $B_{p+1,i}$ ($B_{p+1,i'}$).    After splitting $C_p$ has changed (become larger), so we call it $C_{p,i}$.   Of course, since we split by by removing interstitial bundle between leaves of a measured lamination, we are also obtaining new weight vectors $w_{p,i}$ on $B_{p,i}$.  Now we follow the same procedure for some larger $i$, say $i'$, splitting $C_{p,i'}$ from $B_{p+1,i'}$.   Continuing indefinitely, we obtain a valid inverse sequence $B_{p,i}$ with weight vectors $w_{p,i}$ whose inverse limit is a  $(\PB,\bar\OB)$-measured lamination with levels $p\le j\le q$.   What we have constructed so far is a finite depth (or finite height) multi-level measured lamination which could be expressed in the usual way as $\displaystyle \bigcup_{p\le j\le q}L_j$, where each $L_j$ is a lamination in the complement of higher level laminations.

By induction, we construct a realization of the weight vector on the original $B=B_q$ induced by the measures on the laminations $L_j$ in  $\displaystyle \bigcup_{ j\le q}L_j$.    We are using Theorem \ref{FiniteDepthThm} to move from $(\PB,\bar\OB)$-measured laminations to multilevel measured laminations and vice versa.

In general, there is no upper bound for $q$.  In that case we apply the above argument to the branched manifold $A_q=\cup_{p\le q} C_p$ for some fixed $q$ to obtain a lamination which realizes the weights on $A_q$.  Next, we apply the argument to $A_{q+1}$ to obtain a lamination realizing the weights on $A_{q+1}$ and we see from the uniqueness of the laminations obtained from the construction, that the lamination obtained from $A_{q+1}$ extends the lamination obtained for $A_q$.  Repeating the process indefinitely, we construct the entire lamination realizing weights. \end{proof}

We shall also need a realization theorem for codimension-1 branched manifolds $B\embed M$ with weights in $\PB_r$.   As before, for any $m\in \integers^r$ we can define the associated branched surface $B_m$ as the union of sectors with weights at level $\ge m$, and we can define $C_m$ for every $m\in \integers^r$ such that $B_m=B_{(m_1,m_2,\ldots,m_r)}= B_{(m_1,m_2,\ldots,m_r+1)}\cup C_{(m_1,m_2,\ldots,m_r)}$.  It follows from the locally finite structure of  $B$, that in any compact submanifold of $M$, only finitely many weights at finitely many levels occur.   Although globally the levels have a more complex structure, being indexed by $\integers^r$ with the lexicographical order, on a compact submanifold $\hat M$ of $M$ the weights on $\hat B=B\cap \hat M$ take only finitely many values with finitely many levels.  An example of the type of branched surface of the kind we are analyzing here is shown in Figure \ref{OrderArrayBr}.  In the example, $r=2$ and the branched surface looks like a 2-dimensional array.  When we restrict to a branched submanifold $\hat B$ with its invariant weight vector, since there are only finitely many levels in $\integers^r$ represented by the weights, we are in the setting of the previous theorem and can realize the weights on $\hat B$ by a $(\PB_r,\bar\OB_r)$-measured lamination, with only finitely many levels having non-trivial transverse measures.    Then, exhausting $B$ by a nested sequence of compact branched surfaces, we can use the previous theorem to prove the following theorem:

\begin{thm}  \label{PrOrRealizeThm} Suppose $B$ is a branched manifold (compact or non-compact) embedded in a 3-manifold $M$, and suppose $w$ is an invariant $\PB_r$ weight vector on $B$, with associated branched surfaces $B_m$.   If for every $m$, $C_m=B_m\setminus B_{m+1}$ is smoothly attached to  $B_{m+1}$, then the weight vector $w$ is realizable and the realization is canonical.
\end{thm}

We finish this section with an example which illustrates the potential usefulness transverse measures with values in ordered algebraic structures obtained by mixed insertion.

\begin{example}  \label{MixedExample} Let $S$ be a surface and let $\LB$ be the ordered abelian group $\displaystyle \naturals_0\bigins_{0\le n\le 2} B_n$  where $B_0=[0,\infty]$, $B_1=[0,\infty]$ and $B_2= \bar\naturals_0$.  What are the essential laminations in $S$ with transverse Borel $\LB$-measures with locally finite associated measures?  At the highest level, they have integer coefficients, so they are curve systems.  In the complement of the curve systems, at the next level, we have a measured lamination, and in the complement of the union of the curve system with a measured lamination, we have another measured lamination.  We see that $\LB$-measured laminations have a very particular structure, depending on $\LB$.   \end{example}

\section {Laminations well-covered by  $(\PB_r,\bar\OB_r)$-measured laminations.} \label{WellCoveredSection}

There is a well-known method for constructing more interesting laminations with a transverse structure defined only locally.  Suppose we are given a lamination $L\embed M$, where $M$ is a manifold, and the lift $\tilde L$ to the universal cover $\tilde M$ has a transverse real structure $\mu$.  Further suppose that the action of $\pi_1(M)$ on $\tilde L$ preserves $\mu$, but only up to scalar multiplication.  Thus there is a {\it stretch homomorphism} $\phi: \pi_1(M)\to (\reals_+,\cdot)$ such that for $\gamma\in \pi_1(M)$, the translate $\gamma (\tilde L,\mu)=(\tilde L,\phi(\gamma)\mu)$.  This is called a lamination with a transverse {\it affine structure}, see \cite{UO:Stretch}, \cite{HO:AffineSpaces}, \cite{OP:Affine}, \cite{OP:Broken}.   There are many laminations which do not have transverse $\reals$-measures, but do admit affine structures.  (Usually we deal with codimension-1 laminations.)  The homomorphism $\log( \phi)$ can be interpreted as a cohomology class $\log\phi\in H^1(M;\reals)$.  The group $ (\reals_+,\cdot)$ actually represents the group of order automorphisms of the additive semi-group $(\reals_+,+)$.

We can play the same game with transverse $\LB$ structures, but we want $\LB$ to be a multiplicative group so that we can again define a homomorphism $\phi:\pi_1(M)\to (\LB,\cdot)$.  Thus we would like to use $\LB=\PB$, or $\LB=\PB_r$, the only semifields we have encountered, with multiplicative inverses.  Since $(\PB,\bar\OB)$-measures are determined by their values in $\PB$, we can work with these, or, similarly, we can work with $(\PB_r,\bar\OB_r)$ measures.   In place of the scalar multiplication used in the construction of affine laminations, we use ordered abelian semigroup automorphisms $\psi_\lambda:\bar\OB_r\to\bar\OB_r$ defined by $\psi_\lambda(x)=\lambda x$, $\lambda\in \PB_r$, see Proposition \ref{OrderAutoProp}.

\begin{defn}  A lamination $L\embed M$ is {\it well-covered} by a $(\PB_r,\bar\OB_r)$-measured lamination if there is a transverse $(\PB_r,\bar\OB_r)$-measure $\mu$ for the lift $\tilde L$ of $L$ to the universal cover $\tilde M$  with the property that there exists a homomorphism $\phi:\pi_1(M)\to (\PB_r,\cdot)$ such 
that $\gamma (\tilde L,\mu)=(\tilde L,\phi(\gamma)\mu)$ on compact subsets of transversals.  Thus if $\phi(\gamma)=\lambda\in\PB_r$, then the covering translation moves $\tilde L$ to itself, but applies $\psi_\lambda$ to the transverse measure.
Regarding $\PB_r$ as a product $\integers^r\times (0,\infty)$, the homomorphism $\phi$ yields homomorphisms $\phi_i:\pi_1(M)\to \integers$, $i=1,2,\ldots,r$ by projecting to the $i$-th factor,  which can be interpreted as an element of $H^1(M;\integers)$ and is called a {\it level shift homomorphism}.   In addition, $\phi$ determines a homomorphism and $\phi_\reals:\pi_1(M)\to \reals_+$, called the {\it stretch homomorphism},  and $\log \phi_\reals$ can be interpreted as a cohomology class in $H^1(M;\reals)$.  

If a lamination $L$ is embedded in a manifold $M$ and is carried by a branched manifold $B$, we can always replace $M$ by another smaller manifold, for example a manifold which deformation retracts to $B$.   If $L$ is well-covered by a  $(\PB_r,\bar\OB_r)$-measured lamination in the universal cover of this smaller manifold, we still say $L$ is {\it well-covered by a  $(\PB_r,\bar\OB_r)$-measured lamination}.

\end{defn}

Clearly, the homomorphism $\phi$ in the definition is determined by the $\phi_i$'s and $\phi_\reals$, thus it is determined by $r+1$ cohomology classes.

\begin{exs}  We will give three related examples here, two of them are laminations carried by the same branched surface in a 3-manifold.
In Figure \ref{Local}(a), we see a branched surface $B$ shown immersed in $\reals^3$.  Actually, it can be embedded in $\reals^3$, and we assume it is embedded and that $M$ is a regular neighborhood of $B$.  In the figure, we show weights on $B$, which are values of $\PB$.  We also see a transversely oriented curve $\alpha$ representing a cohomology class, with multiplier $(0,1/2)$ in the transverse direction.   Moving from one side of the cohomology class to the other, the weight is multiplied by $(0,1/2)$ in $\PB$.  The weights and the cohomology class determine a lamination well-covered by an $\bar\OB$-measured lamination.  In fact, this is an affine lamination.   Since the levels of all weights and the multiplier are $0$, the weights could be interpreted as real weights and a real multiplier.  Cutting open the branched surface on the curve representing the cohomology class, the weights on the resulting branched surface $\hat B$ represent an $\reals$-measured lamination.  Glueing with a stretch of 1/2 yields the affine lamination.

\begin{figure}[H]
\centering
\scalebox{1}{\includegraphics{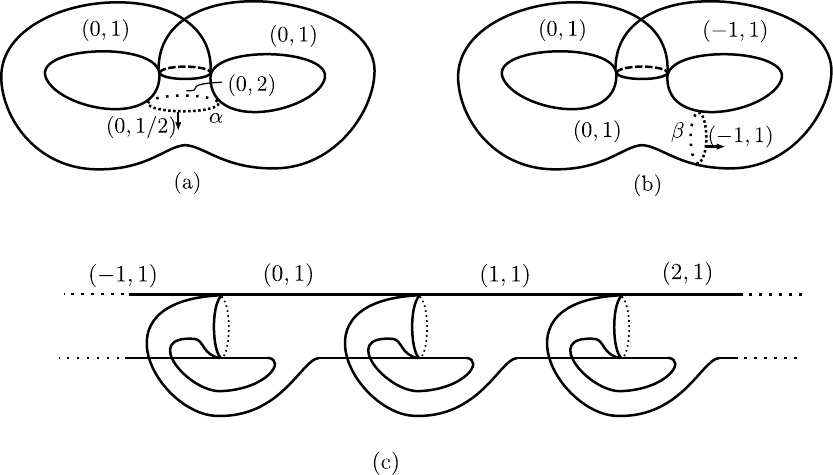}}
\caption{\small Two examples of laminations with transverse $(\PB,\bar\OB)$ structures.}
\label{Local}
\end{figure}

Next we present an example with only a level shift homomorphism, using the same branched surface.  We show weights and the cohomology class at the curve $\beta$ with multiplier $(-1,1)$, but now the multiplier at the cohomology class shifts levels.  The lamination represented by the data has a ``leaf spiraling on itself,"  also sometimes called a spring leaf.  The lamination is completely different from the one in the first example.  In order to see that the data on the branched surface determine a lamination well-covered by a $(\PB,\bar\OB)$-measured lamination we note that the weights on $B$ yield invariant weights on a cover $\tilde B$ as shown in Figure \ref{Local}(c).  According to Theorem \ref{PORealizeThm} this determines a lamination $\tilde L$ carried by $\tilde B$ which is $(\PB,\bar\OB)$-measured.   We need a little more, since we want the lamination to preserved by covering translations {\it and} we want the transverse measure to be transformed according to a homomorphism from the group of deck transformations to $\PB$.   This can be achieved by doing the splittings in the proof of Theorem \ref{PORealizeThm} equivariantly.

For a final example, we change the branched surface $B$ as shown in Figure \ref{LocalMix}.  This branched surface can also be embedded in $\reals^3$ and we assume $M$ is a regular neighborhood in $\reals^3$ of $B$.  We show a transversely oriented curve $\alpha$ with multiplier $(0,1/2)$ representing the stretch homomorphism, and another transversely oriented curve $\beta$ with multiplier $(-1,1)$ representing the level shift homomorphism.  Again the data determine a lamination well-covered by a lamination with a transverse $(\PB,\bar\OB)$ measure.  One must prove the existence of a $(\PB,\bar\OB)$-measured $\tilde L$ carried by some cover $\tilde B$ of $B$ as before.

\begin{figure}[H]
\centering
\scalebox{1}{\includegraphics{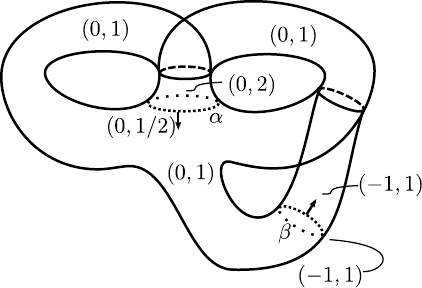}}
\caption{\small Example of a lamination with a local transverse $\PB$ structure.}
\label{LocalMix}
\end{figure}
\end{exs}

Turning to laminations well-covered by $(\PB_r,\bar\OB_r)$- measured laminations, we have the same definitions, but with $r$ level shift homomorphisms and one stretch homomorphism, all of which can be represented by cohomology classes.

\begin{example}  Here is an example of a lamination well covered by a $(\PB_2,\bar\OB_2)$-measured lamination.  In this example, we take the stretch homomorphism to be trivial, so we have two level shift homomorphisms.  Again, we show a branched surface $B$ which can be embedded  in $\reals^3$, and which we assume is so embedded, and we let $M$ be a regular neighborhood of $B$.  This is actually the same branched surface shown in Figure \ref{LocalMix}.  The two level shift homomorphisms are represented by curves with multipliers in $\PB_2$, $\alpha$ with multiplier $(-1,0,1)$ and $\beta$ with multiplier $(0,-1,1)$.

\begin{figure}[H]
\centering
\scalebox{1}{\includegraphics{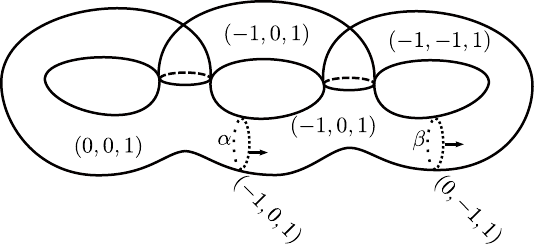}}
\caption{\small Example of a lamination with a local transverse $\PB_2$-structure.}
\label{OrderedP2Ex}
\end{figure}

Figure \ref{OrderArrayBr} shows a cover $\tilde B$ of $B$ which admits an inavariant $\PB_2$ weight vector (with infinitely many entries).  The canonical $(\PB_2,\bar\OB_2)$- measured lamination realizing this weight vector well-covers a lamination carried by $B$ determined by the given data on $B$.  

\begin{figure}[H]
\centering
\scalebox{1}{\includegraphics{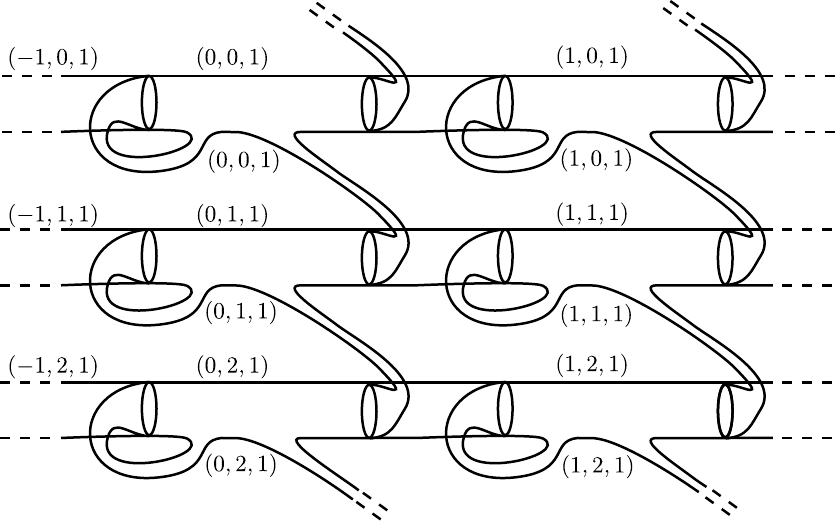}}
\caption{\small Branched surface with invariant weight vector representing lamination which well-covers the lamination in Figure \ref{OrderedP2Ex}.}
\label{OrderArrayBr}
\end{figure}

\end{example}

Suppose $L$ is an essential lamination in a 3-manifold $M$ well-covered by a $(\PB_r,\bar\OB_r)$-measured lamination.  It should be easy to describe an associated action of $\pi_1(M)$ on a $\bar\OB_r$ measured tree.  The action does not necessarily preserve measures on segments; rather, it transforms measures according to stretch and level shift homomorphisms.

\section {Appendix: probability measures.}  \label{Probability}

The author is far from expert in probability.  Therefore it is quite likely that at least some of the ideas presented in this section already exist in some form.   Advice will be appreciated.

If $\LB$ is an ordered abelian semigroup with the lub property, $\LB$-measures can be defined in a reasonable way.   In order to calculate probabilities we need a division operation, so we must work with ordered semifields, which are much less common.  We also want to avoid any kind of infinite measures, at any level, so our choices are even more limited.  Besides $\reals$-measures, one obvious possibility is $\PB=\integers\ins [0,\infty)$.  $\PB$ is a sub-semiring of $\bar\OB$, see Figure \ref{OrderParameter}.  
Notice that $\PB$ does not have the lub property.  For fixed $i\in \integers$ the set $S=\{(i,t):0<t<\infty\}$ is bounded above by $(i+1,1)$, but it does not have a least upper bound. 
\begin{defn} If $(X,\Sigma)$ is a measure space, a {\it probability $\PB$-measure} assigns an element $\nu(E)$ of $\PB$ to each set $E\in \Sigma$ such that the following conditions hold:

(i) $\nu(\emptyset)=0$.

(ii) For each $j$, if $\displaystyle X_j=\bigcup_{\level(\nu(E))=j}E$, then $X_j$ is measurable and $\real(\nu(X_j))=1$ or $\real(\nu(X_j))=0$.

(iii) If $\{E_i\}_{i\in I}$ is a countable collection of disjoint measurable sets in $X$, then 
$$\nu\left(\bigcup_{i\in I} E_i\right )=\sum_{i\in I} \nu(E_i)\in \PB.$$

Often, we assume the probability $\PB$-measure has {\it finite depth}, meaning that the following condition also holds:

(iv) The set $\levels=\{\level(\nu(E)):E\in \Sigma\}$ is finite.
\end{defn}

We observe first that the set $\levels$ is bounded by $\level(\nu(X))$.  Condition (ii) in the definition guarantees that the sum in (iii) can be evaluated and is equal to an element in $\PB$.

Probability $\PB$-measures are useful for calculating relative probabilities of ``black swan" events.  
We prefer a point of view using ``depth" instead of height of levels.  

\begin{defn}
Two probability $\PB$-measures $\nu$ and $\nu'$ are {\it shift equivalent} if $\nu'=(k,1)\nu$ for some $(k,1)\in \PB$, $k\in \integers$.

Any finite depth $\PB$-measure $\nu$ is equivalent by level shift to a measure $\mu$ satisfying $-d\le \level(\nu(E))\le 0$ such that there exist measurable sets $A$ and $B$ such that $\level(\nu(A))=-d$ and $\level(\nu(B))=0$.  Using level omission, we can further modify the measure (while possibly decreasing $d$) such that for each $j$ satisfying $-d\le j\le 0$ there exists $C$ such that $\level(\nu(C))=j$.  When all of these conditions are satisfied, we say $\nu$ is {\it standard}.  We also say that the measure has {\it total depth} $d$.

For a standard probability $\PB$-measure, we say the probability $\nu(E)$ has {\it depth $j$} if \\ $\level(\nu(E))=-j$.  
\end{defn}

Using the depth terminology, we can now explain probabilities of black swan events.  If $\nu(E)$ has depth 0, we can imagine that $\residue(\nu(E))$ represents ordinary probability.  If $\nu(E)$ has greater depth, it has traditional probability 0, but it has a ``higher depth" non-zero probability.  In this way, we assign probabilities
to black swan events whose ordinary probability is 0.  Greater depths correspond to higher ``orders of improbability."

\begin{ex} \label{Dart} (Dartboard example) This is an extremely simple example of a finite depth probability measure.  Suppose we are given a circular dart board $X$ with just a single cross $Y$ drawn on it, the cross consisting of vertical and horizontal diameters.  We can suppose the dart board has area 1, and suppose for simplicity that if $E$ is Lebesgue measurable with measure $\mu(E)$, then the darts have probability $\mu(E)$ of hitting $E$.  Ordinary probability measure assigns the measure (area) of a Lebesgue measurable set $E$ to the event that the dart lands on a point of $E$.  Obviously there are many probability 0 events.  For example, the probability that the dart hits the 1-dimensional cross $Y$ is 0.  We may define a $\PB$-measure $\nu$ as follows.  For any event $E$ with positive measure in the usual sense, we define $\nu(E)=(0,e)$ where $e=\mu(E)$ is the Lebesgue measure of $E$.  If $\mu(E)$ is 0 and $\ell>0$ is the 1-dimensional measure of $E\cap Y$, we define $\nu(E)=(-1,\ell)$.  Otherwise, if $\mu(E)=0$ and $\ell=0$, we define $\nu(E)=0$.  We assume that the total 1-dimensional measure (length) of $Y$ is 1, and that each ray has length 1/4.  The total depth of this $\PB$-measure is 1; it has two levels.  Clearly we could also define a $\PB$-measure with total depth 2 (having three levels) by concentrating the depth two measure at the crossing point at the center of the dart board. 
\end{ex}

The practical usefulness of probability $\PB$-measures comes from the fact that the usual formulas for conditional probability apply and give reasonable answers.

\begin{defn}  If $A$ and $B$ are events in $X$ and $\nu$ is a probability $\PB$-measure on $X$, then the {\it conditional probability of the event $A$ given $B$} is $P(A|B)=\nu(A\cap B)/\nu(B)$.
\end{defn}

\begin{example}[Dartboard example continued]  Suppose $B$ is the event that the dart hits the closed upper half of the dartboard. and $Y$ is the event that the dart hits the cross.  Then $P(Y|B)=\nu(Y\cap B)/\nu(B)=(-1,3/4)/(0,1/2)=(-1,3/4)(0,2)=(-1,3/2)$.  We obtain a probability at depth 1, but this probability is greater than $P(Y)=(-1,1)$.   This says that if the level -1 probability measure on $Y\cap B$ were ``proportionately distributed" with respect to the level 0 measure on $B$, the the relative probability would be $(-1,1)$, but $Y\cap B$ has 3/2 times its share of the measure, so the relative probability is $(-1,3/2)$.

Now suppose $A$ is the event that the dart hits the vertical ray in the upper half of the dart board.  Then $P(A|Y)=(-1,1/4)/(-1,1)=(0,1/4)$, in other words the conditional probability is 1/4 as one would guess.  We can calculate $P(A|B\cap Y)=(-1,1/4)/(-1,3/4)=(0,1/3)$, again as one would expect.  Another less obvious example is $P(A|B)=(-1,1/4)/(0,1/2)=(-1,1/2)$.  If the measure of $A$ were proportionately distributed in $B$, we would have $P(A)=(-1,1/2)$ and $P(A|B)=(-1,1)$, but $A$ is under-represented in $B$.  
\end{example}

Bayes' Theorem holds for probability $\PB$-measures.

\begin{thm} [Bayes Theorem]  Suppose $(X,\Sigma)$ is a measure space and $\nu$ is a probability $\PB$-measure, which we use to calculate probabilities.  Suppose $\{A_i\}$ is a partition of the event space $X$.  Then  $P(B)={\sum _{j}P(B|A_{j})P(A_{j})}$ and 
$$P(A_{i}|B)={\frac  {P(B|A_{i})\,P(A_{i})}{\sum \limits _{j}P(B|A_{j})\,P(A_{j})}}\cdot$$
\end{thm}

\begin{example}[Dart board example continued]  We may suppose that the event space is the dartboard itself.  Let $A_1$ be the closed first quadrant of the dartboard, let $A_2$ be the interior of the second quadrant, let $A_3$ be the closed third quadrant with the center removed, and $A_4$ is the interior of the fourth quadrant.  
$B$ is the event that the dart hits the horizontal line of the cross.   Now we can calculate the various probabilities in Bayes' Formula to calculate $P(A_1|B)$.   For example, $P(B|A_1)=P(B\cap A_1)/P(A_1)=(-1,1/4)/(0,1/4)=(-1,1)$.  Similarly we calculate $P(B|A_2)=0$, $P(B|A_3)=(-1,1/4)$, $P(B|A_4)=0$.   Then Bayes formula gives:
$$P(A_1|B)=\frac {(-1,1)(0,1/4)}{(-1,1)(0,1/4)+(-1,1)(0,1/4)}=\frac{(-1,1/4)}{(-1,1/2)}=(0,1/2).$$

This says the event $P(A_1|B)$ has probability 1/2 in the usual sense, but the probability would be undefined using real-valued probability.
\end{example}

One can use $\PB$-probability distributions to produce a probability $\PB$-measure starting with an arbitrary $\PB$-measure on an event space $X$.  In practice, probably the most useful special case is the case where the distribution has values in $\PB$ and measure on $X$ is a positive $\bar\reals$ measure.  In any case, one needs to define integrals of measurable functions with values in $\PB$, which we do in the following appendix, Section \ref{Integration}.

\begin{defn}  Suppose $(X,\mu)$ is a $\PB$-measure space.  Let $f:X\to \PB$ be a measurable function and let $\nu(E)=\int_E  f d\mu\in \PB$.  Suppose the  following conditions are satisfied:

(i) For every measurable $E\subset X$, $\real(\nu(E))\le 1$;

(ii) If there exists $E$ with $\level(\nu(E)=j$, then there exists $E$ with $\real(\nu(E))= 1$.

Then $f$ is called a probability $\PB$ {\it distribution} (and $\nu$ is a probability $\PB$-measure).
\end{defn}

To finish this section, we ask whether we can define reasonable probability measures with values in other ordered algebraic structures.  Some good candidates are the following:

\begin{defn}  If $(X,\Sigma)$ is a measurable space a {\it probability $\PB_r$-measure} assigns an element $\nu(E)$ of $\PB_r$ to each set $E\in \Sigma$ such that:
\begin{tightenum}
\item $\nu(\emptyset)=0$, 

\item $\nu(E)=(i_1,i_2, \ldots, i_n, s)$ satisfies $i_j\le0$ and $0<s\le 1$. 

\item If there exists measurable $E$ with $\nu(E)=(i_1,i_2, \ldots, i_n, s)$ then there exists measurable $E$  with $\nu(E)=(i_1,i_2, \ldots, i_n, 1)$. 

\item If $\{E_i\}_{i\in I}$ is a countable collection of pairwise disjoint measurable sets in $X$, then 
$$\nu\left(\bigcup_{i\in I} E_i\right )=\sum_{i\in I} \nu(E_i)\in \PB_r.$$
\end{tightenum}

We 
observe that we have already built into our definition a suitable choice of representative by shift equivalence.   

We can also perform alignment operations to ensure there are no ``gaps," and $\nu$ is {\it standard}:

\noindent (iv) \ If for some measurable $E$, $\nu(E)=(i_1,i_2,\ldots, i_n,s)$ for some $s>0$, and if $(0,0,\ldots, 0)>(j_1,j_2,\ldots, j_n)>(i_1,i_2,\ldots, i_n)$ in the lexicographical order, then there exists a measurable set $E'$ with $\nu(E')= (j_1,j_2,\ldots, j_n,t)$ for some $t>0$.   \end{defn}

We note first that $\PB_r$ is a semifield by Lemma \ref{WellDefinedLemma} (c), so the division operation makes probability calculations possible.  Again, although $\PB_r$ does not have the lub property, condition (ii) ensures that partial sums of the sum in (iii) do have a least upper bound, so the sum makes sense.  

\section {Appendix: integration.}  \label{Integration}

If $\KB$ is an ordered abelian semigroups (commutative ordered semiring), we we will in many cases be able to define integrals of real-valued  ($\KB$-valued) functions with respect to $\KB$-measures on measurable spaces.  The definition for non-negative $\bar\reals$-valued functions will be recursive.  To start the recursion, we have  integrals with respect to positive $\bar\reals$-measures.

\begin{defn}  If $(X,\Sigma)$ is a measurable space, $\LB$ is an ordered commutative semiring,  and $f:X\to \LB$ is a function, we say $f$ is {\it measurable} if for every $c\in \LB$, the set $f\inverse(\{x:x<c\})$ is measurable.  

Suppose $\LB$ is an ordered abelian semiring with the least upper bound property and a largest element $\infty$, and $\KB=\integers \ens \LB$ with largest element $\Infty$ (or $\KB=\naturals_0 \ens \LB$).  As before, we deal with $\KB=\integers \ens \LB$ and leave the obvious modifications for $\KB=\naturals_0 \ens \LB$ to the reader.  If $X$ is a measurable space with a $\KB$-measure $\mu$, and integrals of $\bar\reals$- valued measurable functions have been defined for $\LB$-measures on $X$ then we define the integral with respect to the $\KB$-measure $\mu$ as follows.  Let $f$ be a $[0,\infty]$-valued measurable function on a measurable $A\subset X$.  Then, in terms of associated measures, we define
$$\int_A f d\mu=\sum_{k\in \integers}(k,\int_A f d\mu_k),$$
where the sum is in $\KB$ and may have value $\Infty$.  Any summand of the form $(k,0)$ is interpreted as $0\in\KB$.    (The above sum equals one of the summands if the measure has the set $\levels$ of non-trivial levels bounded above.)  Using this definition recursively, starting with $\bar\reals$-measures on $X$ we can define the integral for $\mu$ a $\KB$-measure,  $$\KB=\bar\OB_r=(\integers\ens(\integers\ens\cdots\ens(\integers\ens(\integers\ens [0,\infty]))\cdots)),$$

\noindent and for 
$$\KB=\bar\SB_r=(\naturals_0\ens(\naturals_0\ens\cdots\ens(\naturals_0\ens(\naturals_0\ens [0,\infty]))\cdots)),$$

Suppose $X$ is a measure space with $\KB$-measure $\nu$, $\KB=\bar\OB_r$ (or $\KB=\bar\SB_r$) and suppose $B$ is a measurable set.  We want to define the integral of an $\bar\OB_r$-valued function.  We will deal with $\bar \OB_r$, but the definition is the same for $\bar \SB_r$.  We have already defined the integral of a positive $\bar\reals$-valued measurable function with respect to $\nu$ above, so we have also defined the integral of a $\bar \OB_0$-valued function, where $\bar\OB_0=[0,\infty]\subset \bar\reals$.  To define the integral recursively, suppose we have defined the integral of $\bar \OB_{j-1}$-valued measurable functions on $B$.  Let $g$ be an $\bar\OB_j$-valued function on $B$.   Then let $B_k=\{x:(k-1,\infty)<g(x)\le (k,\infty)\in \bar\OB_j\}$, a measurable set, where $\infty$ is the largest element in $\bar\OB_{j-1}$.   Let  $B_\infty=\{x:g(x)=\Infty\in \bar\OB_j\}$.  Then we define 

$$\int _B g d\nu=\sum_{k\in \integers}\left(k,\int_{B_k} \residue(g) d\nu\right) +\int_{B_\infty}gd\nu,$$

\noindent where the final integral over $B_\infty$ is $\Infty$ if $\nu(B_\infty)>0$, $0$ otherwise.

Note that $\residue(g)$ is an $\bar \OB_{j-1}$-valued measurable function on $B_k$ and we are assuming integrals of such functions have  been defined.  We have now defined $\int _B g d\nu$ for all $\bar \OB_j$-valued functions $g$.   Thus we recursively define $\int _B g d\nu$ for any $\bar\OB_r$-measured function for any $r$.  \end{defn}

\begin{ex}  \label{DiracExample} We define an $\bar\OB$-measure $\nu$ for $\reals^n$, determined by $\nu_0$ and $\nu_{-1}$.  The measure $\nu_0$ is Lebesgue measure.   For $\nu_{-1}$ we use a counting measure which assigns the number of points in a set if the set is countable, and otherwise assigns $\infty$.   This means that $\nu_{-1}$ assigns $\infty$ to a countably infinite set.   Let the Dirac-$\delta$ function $\delta:\reals^n\to \bar\OB$ at $y\in \reals^n$ be defined by 

 $$\delta(x)=\begin{cases} (1,1) & \mbox{if } x=y \\
                  0           & \mbox{otherwise}.
                                            \end{cases}$$

Then $\int_{\{y\}}\delta(x)d\nu=\int_{\reals^n}\delta(x)d\nu=\nu(\{y\})\delta(y)=(-1,1)(1,1)=(0,1)$, and for a ``real-valued function" $f(x)$ with the property that $\level(f(x))=0$, we have $\int_{\{y\}}f(x)d\nu=(-1,1)(0,\real(f(x))=(-1,\real(f(x))$, which is trivial viewed at level 0.  Also for  a real-valued $f$ and any measurable $A$, $\int_Afd\nu= \int_Afd\nu_0$, which is just the Lebesgue integral of $f$.

More generally, if $B$ is a finite or countable set in $\reals^n$ suppose we define

 $$g(x)=\begin{cases} (1,1) & \mbox{if } x\in B \\
                  0           & \mbox{otherwise}.
                                            \end{cases}$$
and suppose $A$ is any measurable set.    Then  $\int_Ag(x)d\nu=\sum_{x\in A\cap B}(-1,1)(1,1)$, which is $(0,n)$ where $n$ is the number of elements in $A\cap B$.   In particular, if $A\cap B$ is countably infinite, we obtain $\int_Ag(x)d\nu=(0,\infty)$.
\end{ex}

The key to the above example is that we have a $\sigma$-subalgebra of countable sets in the $\sigma$-algebra of Lebesgue measurable sets in $\reals^n$.

Recall that we also want to define integrals for $\PB_r$-valued functions with respect to positive real or $\PB_r$-measure.    This makes it possible to consider $\PB_r$ probability distributions.    There is a problem to overcome:   $\PB_r$ does not have the least upper bound property.   On the other hand, $\PB_r\subset \bar\OB_r$ so the integral is defined in the setting of the above definition.

\begin{defn}  Suppose $X$ is a $\PB_r$-measured space with measure $\mu$ and suppose $f:X\to \PB_r$ is a measurable function.   Suppose $A\subset X$ is a measurable set.   Then $f$ is {\it $\PB_r$-integrable} if the integral $\int_A f d\mu\in \PB_r$ where the integral is interpreted as the integral of a $\bar\OB_r$-valued function with respect to a $\bar\OB_r$ measure.
\end{defn}

It is possible to introduce negative values  for  $\LB$-valued functions in an artificial way, and then integrability becomes an issue in another way.

\begin{defn}  Suppose $\LB$ is an ordered abelian semi-algebraic structure (without negative values).  Define the {\it double} of $\LB$ as $\DLB=\{0\}\cup\left[\{+,-\}\times (\LB\setminus \{0\})\right]$.  We define $y=(+,y)$, $-y=(-,y)$, $-(-,y)=y$, and $-(+,y)=(-,y)=-y$.   This becomes an ordered set with the obvious ordering $(-,y)<0$ for all $y\ne 0$, $(+,y)>0$ for all $y\ne 0$, $(+,y)<(+,z)$ and $(-,y)>(-,z)$ if $y<z$.   A function with values in $\DLB$ can be written as $f=f_+-f_-$ where
 $$f_+(x)
=\begin{cases}f(x) & \mbox{if } f(x)\ge 0 \\
                    0 & \mbox{if } f(x)< 0
                                            \end{cases}$$
$$f_-(x)
=\begin{cases}-f(x) & \mbox{if } f(x)\le 0 \\
                    0 & \mbox{if } f(x)> 0
                                            \end{cases}$$
and $f$ is {\it measurable} if $f_+$ and $f_-$ are measurable.

   Suppose now that $\LB=\OB$ or $\LB=\SB$.   If $f$ is a $\DLB$-valued function on an $\LB$-measure space $X$ with measure $\nu$, and $A\subset S$ is measurable, let $P=\int_Af_+d\nu$, $N=\int_Af_-d\nu$, the integrals of the positive and negative parts of $f$.  We define
 $$\int_Afd\nu=\begin{cases}P=\int_Af_+d\nu & \mbox{if } \level(P)>\level(N) \\
                    N=-\int_Af_-d\nu & \mbox{if } \level(P)<\level(N)\\
                    (\level(P),\real(P)-\real(N)) & \mbox{if } \level(P)=\level(N).
                                            \end{cases}$$
\end{defn}

The above definition is awkward because $\DLB$ does not have a subtraction operation.  Algebraically, $\DLB$ is defective;  if one attempts to give it a structure as an additive group, the associative law fails.   However, we are using an artificial sum of a negative and a positive element defined in $\DSB$ or $\DOB$ by 
 $$(i,s)+(-(j,t))=\begin{cases} (i,s) & \mbox{if } i>j \\
                   -(j,t) & \mbox{if } i<j\\
                    (i,s-t) & \mbox{if } i=j \mbox{ and } s\ne t\\
                    0&\mbox {if } i=j \mbox{ and } s= t.
                                            \end{cases}$$
                                            
\noindent  One can define finite sums of elements in $\DLB$ in a similar way by adding only the elements with maximal level.  An example of the failure of associativity is the following:  $-(0,1)+[(0,1)+(-1,1)]=-(0,1)+(0,1)=0$ but   $[-(0,1)+(0,1)]+(-1,1)=0+(-1,1)=(-1,1)$.  If we do the same addition by simply ignoring all but the maximal level terms, we obtain a well-defined sum.

As with real-valued measures, if $\LB$ is a suitable ordered abelian semiring, then different $\LB$-measures may be related by a Radon-Nikodym derivatives.  Thus, if for all measurable $E$ in $X$ we have $\nu(E)=\int_E f(x) d\mu$, where $f$ is an $\LB$ -valued function, 
we can say that $f$ is a {\it Radon-Nikodym derivative $d\nu/d\mu$}.

\hop\hop

\bibliographystyle{amsplain}
\bibliography{ReferencesUO3}
\end{document}